\documentclass[11pt]{article}

\usepackage[paperwidth=8.5in,paperheight=11in,portrait,top=1in,bottom=1.in,left=1.15in,right=1.15in]{geometry}
\usepackage[utf8]{inputenc}
\usepackage{authblk}
\usepackage{blkarray}

\usepackage{amsfonts,amsmath,amssymb,amsthm}
\usepackage{latexsym,mathrsfs,mathtools,bm}
\usepackage{graphicx,subcaption,epsfig,caption,float,xcolor}
\usepackage{enumitem}
\usepackage{chngcntr}

\usepackage{tikz}
\usetikzlibrary{calc}
\usetikzlibrary{shapes.geometric}

\usepackage{comment}

\usepackage[hidelinks]{hyperref}
\usepackage{bookmark}

\theoremstyle{plain}
\newtheorem{thm}{Theorem}[section]

\newtheorem{lem}[thm]{Lemma}
\newtheorem{prop}[thm]{Proposition}
\newtheorem{coro}[thm]{Corollary}

\newtheorem{defi}[thm]{Definition}
\newtheorem{rem}[thm]{Remark}

\numberwithin{equation}{section}

\newcommand{\cF}{\mathcal{F}}
\newcommand{\cb}{\mathtt{b}}
\newcommand{\cB}{\mathcal{B}}
\newcommand{\tcB}{\widetilde{\mathcal{B}}}
\newcommand{\tcb}{\widetilde{\mathtt{b}}}
\newcommand{\cc}{\mathtt{c}}
\newcommand{\cC}{\mathcal{C}}
\newcommand{\tcC}{\widetilde{\mathcal{C}}}
\newcommand{\tcc}{\widetilde{\mathtt{c}}}
\newcommand{\cA}{\mathcal{A}}
\newcommand{\cD}{\mathcal{D}}

\newcommand{\cM}{\mathcal{M}}

\newcommand{\cX}{\mathcal{X}}
\newcommand{\cY}{\mathcal{Y}}

\newcommand{\A}{\mathcal{A}_q(n,\ell,m)}

\newcommand{\tT}{\widetilde{T}}
\newcommand{\tK}{\widetilde{K}}
\newcommand{\tE}{\widetilde{E}}
\newcommand{\tQ}{\widetilde{Q}}
\newcommand{\tU}{\widetilde{U}}

\newcommand{\CC}{\mathbb{C}}
\newcommand{\II}{\mathbb{I}}
\newcommand{\JJ}{\mathbb{J}}
\newcommand{\NN}{\mathbb{N}}
\newcommand{\FF}{\mathbb{F}}

\newcommand{\bx}{\mathbf{x}}
\newcommand{\by}{\mathbf{y}}

\newcommand{\qbinom}[2]{\left[\begin{array}{c} #1\\ #2\end{array} \right]}

\newcommand{\vx}{\hat{\mathbf{x}}}
\newcommand{\vy}{\hat{\mathbf{y}}}

\newcommand{\po}{\preceq}
\newcommand{\poo}{\preceq_{(\alpha,\beta)}}

\def\brad{0.15} 
\def\srad{0.08} 

\numberwithin{equation}{section}

\title{\bf Bivariate $P$- and $Q$-polynomial structures of the association schemes based on attenuated spaces }

\renewcommand*{\Affilfont}{\normalsize\small}

\author[1]{Pierre-Antoine Bernard}
\author[2]{Nicolas Cramp\'e}
\author[3]{Luc Vinet}
\author[4]{Meri Zaimi}
\author[5]{Xiaohong Zhang\vspace{.5em}}

\affil[1,3,4,5]{Centre de Recherches Math\'ematiques, Universit\'e de Montr\'eal, \newline\vspace{.9em}
P.O. Box 6128, Centre-ville Station, Montr\'eal (Qu\'ebec), H3C 3J7, Canada.}

\affil[2]{Institut Denis-Poisson CNRS/UMR 7013 - Universit\'e de Tours - Universit\'e d'Orl\'eans, \newline\vspace{.9em}
Parc de Grandmont, 37200 Tours, France.}

\affil[2]
{ Laboratoire d'Annecy-le-Vieux de Physique Th\'eorique LAPTh, Universit\'e Savoie Mont Blanc,\newline\vspace{.9em}
 CNRS, F-74000 Annecy,
 France.}

\affil[3]{IVADO, Montr\'eal (Qu\'ebec), H2S 3H1, Canada. \vspace{.9em}}

{
	\makeatletter
	\renewcommand\AB@affilsepx{: \protect\Affilfont}
	\makeatother
	\affil[ ]{E-mail addresses}
	\makeatletter
	\renewcommand\AB@affilsepx{, \protect\Affilfont}
	\makeatother
	\affil[1]{pierre-antoine.bernard@umontreal.ca}
	\affil[2]{crampe1977@gmail.com}
	\affil[3]{luc.vinet@umontreal.ca}\affil[4]{meri.zaimi@umontreal.ca}\affil[5]{xiaohong.zhang@umontreal.ca}
}

\begin{document}

\date{\today} 
\maketitle
\vspace{15mm}
\begin{center}
\textbf{Abstract}\vspace{5mm}\\
\begin{minipage}{13cm}
The bivariate $P$- and $Q$-polynomial structures of association schemes based on attenuated spaces are examined using recurrence and difference relations of the bivariate polynomials which form the eigenvalues of the scheme. These bispectral properties are obtained from contiguity relations of univariate dual $q$-Hahn and affine $q$-Krawtchouk polynomials. The bispectral algebra associated to the bivariate polynomials is investigated, as well as the subconstituent algebra of the schemes. The properties of the schemes are compared to those of the non-binary Johnson schemes through a limit.
\end{minipage}
\end{center}

\medskip

\begin{center}
\begin{minipage}{13cm}

\textbf{MSC2020 database:} 05E30; 33D45; 33D50
\end{minipage}
\end{center}

\vspace{15mm}
\newpage

\section{Introduction}

An $N$-class symmetric association scheme on a set $X$ is a partition of 
$X\times X$ into binary relations $\{R_i\ | \ i = 0,1, \dots, N\}$, which satisfies certain properties. 
The adjacency matrix $A_i$ associated to $R_i$ is a 01-matrix of size $|X|\times |X|$, with rows and columns indexed by elements of $X$ and verifies $(A_i)_{a,b}=1$ if and only if $(a,b)\in R_i$. 
In terms of these adjacency matrices $A_0,\ldots, A_N$, these properties are formulated as
\begin{enumerate}
\item[(i)] $A_0 = \II$, where $\II$ is the $|X| \times |X|$  identity matrix;
    
\item[(ii)] $\sum_{i=0}^{N} A_i = \JJ$, where $\JJ$ is the $|X| \times |X|$ all-ones matrix;

\item[(iii)] $A_i^\top = A_i$ for $i = 0, 1, \ldots, N$, where $\top$ stands for the transposition;

\item[(iv)] The following relations hold: $A_i A_j = A_j A_i = \sum_{k=0}^{N} p_{ij}^k A_k$,
\end{enumerate}
where $p_{ij}^k$ are constants, called intersection numbers of the scheme. More details can be found in \cite{BI,God2,God}.

Over the years, this construction has demonstrated its potency in establishing connections with various fields in algebra, combinatorics, and the study of special functions. Pioneer works \cite{Del,BI} have shown that schemes whose intersection numbers verify a metric condition are in correspondence with distance-regular graphs \cite{BCN} and are associated to orthogonal polynomials. Such schemes are said to be \textit{$P$-polynomial}. Similarly, schemes whose Krein parameters verify a similar metric condition are said to be \textit{$Q$-polynomial}. A noteworthy advancement was realized by Leonard \cite{Leo} who demonstrated that association schemes satisfying both $P$- and $Q$-polynomial properties could be described by the families of hypergeometric orthogonal polynomials of the terminating branch of the Askey scheme \cite{Koek}. Hence, $P$- and $Q$-polynomial association schemes such as the Hamming scheme or the Johnson scheme are characterized respectively by Krawtchouk and dual Hahn polynomials. These polynomials satisfy both a three-term recurrence relation and a three-term difference relation, and are solutions to a bispectral problem. Leonard's theorem has also been understood algebraically through the introduction and classification of Leonard pairs \cite{Ter,TerClas,TerNote}.
In this algebraic interpretation, the two bispectral operators (that is, the recurrence and difference operators) associated to a family of orthogonal polynomials of the Askey scheme, as well as the two generators of the subconstituent (or Terwilliger) algebra \cite{ter1,ter2,ter3} of a $P$- and $Q$-polynomial association scheme, are represented by a Leonard pair on certain (finite-dimensional) irreducible modules.  

As a mean to extend these correspondences, recent investigations have considered the connection between symmetric association schemes, multivariate polynomials \cite{BCPVZ,BKZZ} and $m$-distance-regular graphs \cite{BCVZZ}. This led to the definition of bivariate $P$- and $Q$- polynomial association schemes of type $(\alpha,\beta)$ \cite{BCPVZ} and multivariate $P$- and $Q$- polynomial association schemes \cite{BKZZ}. Many schemes were shown to fit these definitions, such as the  direct product of $P$- or $Q$- polynomial association schemes,  the symmetrization of an association scheme of certain length (also called a generalized Hamming scheme), the non-binary Johnson scheme and schemes based on attenuated or isotropic subspaces. The bivariate $P$- and $Q$-polynomial structure of type $(\alpha,\beta)$ of the non-binary Johnson scheme was investigated in details in \cite{CVZZ}. The $P$- and $Q$-polynomial structure in the sense of \cite{BKZZ} of the non-binary Johnson scheme and those based on attenuated spaces was examined in \cite{BKZZ2}, where the Krein parameters of the scheme were obtained from a formula based on convolutions of the dual eigenvalues. 

This paper further examines the association schemes based on attenuated spaces and puts forward an alternative approach to deriving its bivariate polynomial structures based on the recurrence and difference relations satisfied by the eigenvalues of the scheme. This approach is similar to the one adopted in \cite{CVZZ} where the non-binary Johnson scheme is treated, but the techniques in the present paper benefit from recent developments on bivariate polynomials which can be expressed in terms of univariate polynomials satisfying recurrence, difference and contiguity relations \cite{CFR,CFGRVZ,CFGRVZ2,CZ}. The work presented in \cite{CZ} is particularly relevant, as it studies a class of rank 2 Leonard pairs \cite{IT}, called factorized $A_2$-Leonard pairs, which appear in the context of association schemes with bivariate polynomial structures of certain type. In terms of the notions introduced in \cite{BCPVZ} and further examined in \cite{BCVZZ}, these association schemes involve recurrence relations of type $(\alpha,\beta)$ with either $\alpha=0$ or $\beta=0$, implying that one of the recurrence relations has three terms. The non-binary Johnson scheme and association schemes based on attenuated spaces both are of these types and are related to factorized $A_2$-Leonard pairs. Another purpose of this paper is to explore the algebras which arise from the association schemes based on attenuated spaces, that is the bispectral algebra realized by the recurrence and difference operators of the associated bivariate polynomials, and the subconstituent algebra of the scheme. Finally, we also consider in this paper the limit $q \rightarrow 1$ and discuss the relation between schemes based on attenuated spaces and the non-binary and binary Johnson schemes.

The paper is organized as follows. Section \ref{sec:attenuatedschemes} recalls the definition of association schemes based on attenuated spaces and some of their properties, such as the eigenvalues and dual eigenvalues. In Section \ref{sec:bispectral}, the bispectral properties (that is the recurrence relations and difference equations) of the eigenvalues are obtained, and the associated bispectral algebra is discussed. The relevant definitions concerning bivariate $P$- and $Q$-polynomial association schemes are presented in Section \ref{sec:bivariatePQpoly}. The bivariate $P$-polynomial structure for association schemes based on attenuated spaces is determined in Section \ref{sec:biPpoly}, and similarly the bivariate $Q$-polynomial structure is determined in Section \ref{sec:biQpoly}. The subconstituent algebra of the schemes is discussed in Section \ref{sec:subconstituent}. The $q\to 1$ limit and connections to the non-binary Johnson scheme are the objects of Section \ref{sec:limit} and Section \ref{sec:outlook} consists of concluding remarks. Appendix \ref{app:notations} gathers some useful and standard notations concerning polynomials of ($q$-)hypergeometric type. Moreover, relevant recurrence, difference and contiguity relations of univariate polynomials are collected in Appendix \ref{app:relations}.

\section{Association schemes based on attenuated spaces}\label{sec:attenuatedschemes}
In this section, we recall the main properties of association schemes based on attenuated spaces. For more details, the reader is encouraged to consult \cite{Kur,ter90poset,WGL,Wen16}. 

Let $\FF_q^{n+\ell}$ be the vector space of dimension $n+\ell$ over the finite field $\FF_q$, with $n$ a positive integer, $\ell$ a non-negative integer and $q$ a prime power, and consider a fixed subspace $\mathbf{w} \subset \FF_q^{n+\ell}$ of dimension $\ell$. The collection of all subspaces $\mathbf{x}\subseteq \FF_q^{n+\ell}$ such that $\mathbf{x}\cap \mathbf{w}=0$ is the attenuated space associated to $\mathbf{w}$. 
Let $m$ be a non-negative integer ($0\leq m \leq n$) and $X$ be the collection of all $m$-dimensional subspaces in the attenuated space associated to $\mathbf{w}$. 
It is shown in \cite{WGL} that the set $X$ together with the relations 
\begin{equation}
R_{ij} = \{ (\mathbf{x},\mathbf{y}) \in X \times X \ | \ \dim((\mathbf{x}+\mathbf{w})/\mathbf{w} \cap (\mathbf{y}+\mathbf{w})/\mathbf{w} ) = m - i \ \text{and} \ \dim(\mathbf{x} \cap \mathbf{y}) = m-i-j \}
\end{equation}
defines a symmetric association scheme that we will denote by $\A$. The non-empty relations $R_{ij}$ are those for which the pairs $(i,j)$ belong to the set
\begin{equation}\label{eq:domain}
\cD = \{ (a,b) \in \NN^2 \ | \ a+b \leq m, \ a \leq n-m, \ b\leq \ell   \} \,,
\end{equation}
where $\mathbb{N}^2$ denotes the set of pairs of non-negative integers.
The cardinality of the vertex set $X$ of the scheme $\A$ is given by 
\begin{equation}\label{eq:cardinality}
|X|=q^{m\ell}\qbinom{n}{m},
\end{equation}
where $\qbinom{n}{m}$ is the $q$-binomial coefficient (also called Gaussian coefficient, see Appendix \ref{app:notations}).  

The adjacency matrices $A_{ij}$ for $(i,j)\in \cD$ of the scheme are $|X|\times |X|$ matrices, with rows and columns indexed by elements of $X$ such that 
\begin{equation}
(A_{ij})_{\bx\by} = 
\begin{cases}
1 & \text{if } (\bx,\by) \in R_{ij}\,,\\
0 & \text{otherwise} \,.
\end{cases}
\end{equation}
It can be verified from the definition of the scheme that $A_{00}$ is the identity matrix. The Bose--Mesner algebra $\cM$ is defined as the matrix algebra over $\CC$ generated by $\{A_{ij}\}_{(i,j) \in \cD}$ under the usual matrix product. The adjacency matrices $A_{ij}$ form a basis of $\cM$ with  
\begin{equation}\label{eq:Amnij}
A_{mn}A_{ij} = \sum_{(a,b) \in \cD} p_{mn,ij}^{ab}A_{ab}\,,
\end{equation}
where $p_{mn,ij}^{ab}$ are the intersection numbers. As for any association scheme, there exists a set of primitive idempotents $E_{rs}$ labeled by pairs $(r,s)$ in a subset $\cD^\star \subset \NN^2$, with $|\cD^\star|=|\cD|$, $(0,0) \in \cD^\star$ and $E_{00}=\frac{1}{|X|}\JJ$, such that
\begin{align}
A_{ij} = \sum_{(r,s) \in \cD^\star} T_{ij}(r,s) E_{rs}\,, \quad E_{rs} = \frac{1}{|X|}\sum_{(i,j) \in \cD} U_{rs}(i,j) A_{ij}\,, \label{eq:eigenvalues}
\end{align}
where $T_{ij}(r,s)$ are the eigenvalues and $U_{rs}(i,j)$ are the dual eigenvalues of the scheme. Here we take $\cD^\star=\cD$, but for clarity we keep the notation $\cD^*$ as this subset may differ in general. The primitive idempotents $E_{rs}$ form another basis of the Bose--Mesner algebra $\cM$, and they satisfy
\begin{equation}
\sum_{(r,s)\in \cD^\star} E_{rs} = \II\,, \qquad E_{ij}E_{rs} = \delta_{ir}\delta_{js}E_{rs}\,.
\end{equation}
The eigenvalues and dual eigenvalues are related by the Wilson duality:
\begin{equation}\label{eq:WilsDual}
\frac{U_{rs}(i,j)}{U_{rs}(0,0)} = \frac{T_{ij}(r,s)}{T_{ij}(0,0)}\,.
\end{equation}
The algebra $\cM$ is in fact also closed under the Schur product $\circ$ (or entrywise product), therefore there exist scalars $q_{mn,ij}^{ab}$ such that 
\begin{equation}\label{eq:Emnrs}
E_{mn} \circ E_{rs} = \frac{1}{|X|} \sum_{(a,b) \in \cD^\star} q_{mn,rs}^{ab} E_{ab} \,.
\end{equation}
The numbers $q_{mn,rs}^{ab}$ are known as the Krein parameters of the scheme. The adjacency matrices $A_{ij}$ are Schur idempotents since
\begin{equation}
\sum_{(i,j)\in \cD} A_{ij} = \JJ\,, \qquad A_{ij}\circ A_{rs} = \delta_{ir}\delta_{js}A_{rs}\,.
\end{equation}


Tarnanen, Aaltonen, and Goethals developed a method to prove a given collection $\cA=\{A_0=I,A_1, \ldots, A_d\}$ of symmetric 01-matrices on a set $X$ form an association scheme \cite{TAG}. It says if there exists a collection of matrices $\{C_0=J, C_1,\ldots, C_d\}$ lying in the complex vector space spanned by $\cA$, such that 
$A_i=\sum_{j=0}^d \alpha_i(j) C_j$ and 
$C_rC_s=\sum_{t=0}^{\min\{r,s\}}\beta_{r,s}(t) C_t$ for some scalars $\alpha_i(j)$ and $\beta_{r,s}(t)$, 
then $\cA$ forms an association scheme and expressions of the eigenvalues of the scheme in terms of $\alpha_i(j)$ and $\beta_{r,s}(t)$ are provided. They further provide a set of matrices $C_0=J,\ldots, C_d$ satisfying the above conditions for the non-binary Johnson schemes and obtain the eigenvalues of the scheme. 
Kurihara \cite{Kur} used a different realization of $\A$ rather than the one based on attenuated spaces. In fact, let $V$ and $E$ be $n$- and $\ell$-dimensional vector spaces over $\FF_q$, respectively. Consider the set $X_m=\{(U,\xi)\, | \, U\subseteq V, \dim(U)=m \text{ and } \xi\in L(V,U)\}$, where $L(V,U)$ denotes the set of all linear maps from $V$ to $U$. By use of the binary relations of the Grassmann scheme $J_q(n,m)$ on $V$ and the binary relations of the bilinear form schemes, Kurihara defines a binary relation on $X_m$ and obtains a different realization of $\A$.   
By using a similar approach as in \cite{TAG}, constructing the matrices $C_j$ and calculating the scalars $\alpha_i(j)$ and $\beta_{r,s}(t)$,  
 all the eigenvalues of $\A$ are obtained, expressed in terms of the $q$-Eberlein polynomials and $q$-Krawtchouk polynomials, in correspondence to the Grassmann scheme and bilinear form scheme realization.

Indeed, let us first define the following univariate polynomials for $0\leq k \leq \min(N,\ell)$
\begin{align}
 &K_k(N,\ell;q;x)= (q^{-\ell};q)_k \qbinom{N}{k} q^{\ell k}  K_k^{aff}(q^{-x};q^{-\ell-1},N;q)\,, \label{eq:affKrawscheme}
\end{align}
and for $0 \leq k \leq \min(N-m,m)$
\begin{align}
 &E_k(N,m;q;x)=  q^{k^2}\qbinom{m}{k} \qbinom{N-m}{k} R_k(\mu(x);q^{-m-1},q^{-N+m-1},N-m|q)\,, \label{eq:dHahnscheme} \\
 &Q_k(N,m;q;x) = \left( \qbinom{N}{k} - \qbinom{N}{k-1} \right) R_x(\mu(k);q^{-m-1},q^{-N+m-1},N-m|q)\,, \label{eq:Hahnscheme}
\end{align}
where (using the standard notations of \cite{Koek}, which are recalled in the Appendix \ref{app:notations}) $K_k^{aff}(q^{-x})$ is the affine $q$-Krawtchouk polynomial, $R_k(\mu(x))$ is the dual $q$-Hahn polynomial and $\mu(x)= q^{-x}+q^{x-N-1}$. Note that $R_x(\mu(k))$ corresponds to the $q$-Hahn polynomial. For $k$ outside the domain, we will use the convention that the polynomials $K_k$, $E_k$ and $Q_k$ vanish. With these notations, the eigenvalues and dual eigenvalues of the association schemes $\A$ are given by\footnote{These correspond to $P_{ij}(r+s,s)$ and $U_{r+s,s}(i,j)$ in the notations of \cite{Kur}.}, for $(i,j)\in \cD$ and $(r,s)\in \cD^\star$,
\begin{align}
&T_{ij}(r,s)= q^{i\ell} K_j(m-i,\ell;q;s) E_i(n-s,m-s;q;r)\,, \label{eq:Tij} \\
&U_{rs}(i,j)=\frac{\qbinom{n}{m}}{\qbinom{n-s}{m-s}} K_s(m-i,\ell;q;j) Q_r(n-s,m-s;q;i)\,. \label{eq:Urs}
\end{align}

Up to factors, the bivariate polynomials \eqref{eq:Tij} and \eqref{eq:Urs} both have the structure of an intricate product of a pair of polynomials of the $q$-Askey scheme, where the parameters of one depend in a specific way on the degree or the variable of the other. Indeed, in terms of basic hypergeometric functions, one can write    
\begin{align}
&T_{ij}(x,y) \propto \qbinom{m-y}{i}
{}_3\phi_2 \Biggl({{q^{-j},\; q^{-y},\;0}\atop
{q^{-\ell},\; q^{-m+i}}}\;\Bigg\vert \; q,q\Biggr)
{}_3\phi_2 \Biggl({{q^{-i},\;q^{-x}}, \; q^{x+y-n-1} \atop
{q^{-n+m},\; q^{-m+y} }}\;\Bigg\vert \; q,q\Biggr),\\
&U_{ij}(x,y) \propto \qbinom{m-x}{j}{}_3\phi_2 \Biggl({{q^{-i},\;q^{-x}}, \; q^{i+j-n-1} \atop
{q^{-n+m},\; q^{-m+j} }}\;\Bigg\vert \; q,q\Biggr){}_3\phi_2 \Biggl({{q^{-j},\; q^{-y},\;0}\atop
{q^{-\ell},\; q^{-m+x}}}\;\Bigg\vert \; q,q\Biggr),
\end{align}
where $\propto$ refers to proportionality up to a function $f(i,j)$ which does not depend on $x,y$. Such bivariate polynomials are said to be of Tratnik type in reference to \cite{Tra}. They  are shown in \cite{CZ} to appear as the entries of change of basis matrices in the study of factorized $A_2$-Leonard pairs, which form a special class of rank 2 Leonard pairs \cite{IT}. The most general bivariate polynomials obtained in \cite{CZ} involve a product of $q$-Hahn and dual $q$-Hahn polynomials (see equation (7.11) in \cite{CZ}). These bivariate polynomials are shown to satisfy two recurrence relations and two difference equations; in both cases (recurrence or difference), one of the equations contains three terms and the other contains at most seven terms (which are related to the root system of the Weyl group $A_2$). The eigenvalues $T_{ij}(r,s)$ and dual eigenvalues $U_{rs}(i,j)$ belong to the framework studied in \cite{CZ}: they both correspond to one of the limit cases given in that paper. This means that they satisfy bispectral properties which are compatible with the bivariate polynomial structures of association schemes. In the rest of this paper, we will obtain these properties explicitly and directly from the expressions \eqref{eq:Tij}--\eqref{eq:Urs}. The bispectral properties will then allow to identify the bivariate polynomial structures of the association scheme. 

Note that when $\ell=0$, the scheme $\A$ is isomorphic to the Grassmann scheme, and when $m=n$, it is isomorphic to the bilinear forms scheme, both of which are known to be both (univariate) $P$- and $Q$-polynomial. 
In these two cases, all the points in the domain $\cD$ in \eqref{eq:domain} lie on a horizontal or a vertical line, respectively. 
It follows that exactly one of the two sets of eigenvalues $\{T_{i0}\}_{i=0}^{\min(m,n-m)}$ or $\{T_{0j}\}_{j=0}^{\min(m,\ell)}$ is defined, and similarly for the dual eigenvalues. These eigenvalues and dual eigenvalues satisfy a three-term recurrence relation. 
Here we focus on the case when $m>n$ and $\ell>0$.

\section{Bispectral properties of the eigenvalues} \label{sec:bispectral}

In this section, we compute recurrence relations and difference equations for the bivariate polynomials $T_{ij}(r,s)$. These relations are referred to as the bispectral properties of $T_{ij}(r,s)$. 

\subsection{Recurrence relations}
We start by obtaining equations which relate polynomials $T_{ij}(r,s)$ with different degrees $(i,j)$.
\begin{prop}\label{prop:recT}
Let $\lambda(s)=q^{-s}-1, \, \Lambda(r,s)=(q^{-r}-1)(1-q^{r+s-n-1})$, then the polynomials $T_{ij}(r,s)$ satisfy the following recurrence relations for all $(i,j)\in\cD$ and $(r,s) \in \cD^\star$
\begin{subequations}
\begin{align}\label{eq:recT1}
 \lambda(s) T_{ij}(r,s)=& 
 \sum_{\epsilon\in \{0,+1,-1\} }  \cb^{\epsilon}(i,j)\ T_{ij+\epsilon}(r,s) \,,
\end{align}
where 
\begin{align}
& \cb^{+}(i,j) = q^{i+j-\ell-m}(q^{j+1}-1)\, ,\\
&\cb^{-}(i,j)=(q^{j-\ell-1}-1)(q^{i+j-m-1}-1)\, ,\\
& \cb^{0}(i,j)=
-\cb^{+}(i,j-1)-\cb^{-}(i,j+1)\,,
\end{align}
\end{subequations}
and 
\begin{subequations}
\begin{align}\label{eq:recT2}
 q^{-s}\Lambda(r,s) T_{ij}(r,s)=& \sum_{\epsilon,\epsilon' \in \{0,+1,-1\}}  \cc^{\epsilon,\epsilon'}(i,j)\ T_{i+\epsilon,j+\epsilon'}(r,s)  \,,
\end{align}
where
\begin{align}
& \cc^{+0}(i,j) = q^{j-\ell-n-1}(1-q^{i+1})^2\, ,\\
& \cc^{+-}(i,j) = q^{-n-1}(1-q^{i+1})^2 (1-q^{j-\ell-1})\, , \\
& \cc^{0+}(i,j)=q^{i+j-\ell-m}(q^{i-n+m}-1) (q^{j+1}-1)\, , \\
& \cc^{00}(i,j)=-(1-q^{i-n+m})(1-q^{i-m})-q^{-n-1}(1-q^i)^2-(1-q^{i-n+m})\cb^0(i,j)\, , \\
& \cc^{0-}(i,j) =-(1-q^{i-n+m})(1-q^{j-\ell-1})(1-q^{i+j-m-1})\, , \\
& \cc^{-+}(i,j)= q^{i-m-1}(1-q^{i-1-n+m})(q^{j+1}-1)\, , \\
&\cc^{-0}(i,j) = (1-q^{i-1-n+m})(1-q^{j-m+i-1})\,,\\
&\cc^{++}(i,j)=\cc^{--}(i,j)=0\,.
\end{align}
\end{subequations}
\end{prop}
\begin{proof}
Write the eigenvalues $T_{ij}(r,s)$ as in \eqref{eq:Tij}. Note that for $i+s>m$, we have $E_i(n-s,m-s,q,r)=0$ and hence also $T_{ij}(r,s)=0$. Equation~\eqref{eq:recT1} follows from applying the recurrence relation \eqref{eq:crecK} with $\epsilon=0$ to the polynomial $K_j(m-i,\ell;q;s)$. 

To prove \eqref{eq:recT2}, use the recurrence relation \eqref{eq:recE} on the polynomial $E_i(n-s,m-s,q,r)$ to find
\begin{align}\label{eq:recT21}
    &q^{-s}\Lambda(r,s)T_{ij}(r,s)
    =q^{i\ell}(1-q^{i+1})^2q^{-n-1}K_j(m-i,\ell; q;s)E_{i+1}(n-s,m-s;q;r)\nonumber \\
    &\qquad \qquad \qquad -q^{i\ell}\big((1-q^i)^2q^{-n-1}+(1-q^{i-n+m})(1-q^{i-m+s})q^{-s}\big)K_j(m-i,\ell; q;s)E_{i}(n-s,m-s;q;r)\nonumber \\
    &\qquad \qquad \qquad +q^{i\ell}(1-q^{i-1-n+m})(q^{-s}-q^{i-m-1})K_j(m-i,\ell; q;s)E_{i-1}(n-s,m-s;q;r)\, . 
\end{align}
The contiguity recurrence relation \eqref{eq:crecK} with $\epsilon=-$ can be used to rewrite the first term in the R.H.S. of \eqref{eq:recT21} in terms of polynomials $K_j$ with shifted parameter $i \to i+1$. Similarly, the contiguity difference relation \eqref{eq:crecK} with $\epsilon=+$ can be used to rewrite the third term in the R.H.S. of \eqref{eq:recT21} in terms of polynomials $K_j$ with shifted parameter $i \to i-1$. Finally, the recurrence relation \eqref{eq:crecK} with $\epsilon=0$ can be used on the second term in \eqref{eq:recT21}. Equation \eqref{eq:recT2} is then obtained by identifying the polynomials $T_{i+\epsilon,j+\epsilon'}(r,s)$ with the help of \eqref{eq:Tij}.   
\end{proof}

\subsection{Difference equations}

We now obtain equations which relate polynomials $T_{ij}(r,s)$ at different values of the variables $(r,s)$.
\begin{prop}\label{prop:diffT}
Let $\theta_i=q^{-i}-1$, then the polynomials $T_{ij}(r,s)$ satisfy the following difference equations for all $(i,j) \in \cD$ and $(r,s) \in \cD^\star$
\begin{subequations}
\begin{align}\label{eq:diffT1}
 \theta_i T_{ij}(r,s)=& \sum_{\epsilon\in \{0,+1,-1\} }  \cB^{\epsilon}(r,s)\ T_{ij}(r+\epsilon,s) \,,
\end{align}
where 
\begin{align}
& \cB^{+}(r,s)=\frac{(1-q^{r+m-n})(1-q^{r+s-m})(1-q^{r+s-n-1})}{(1-q^{2r+s-n-1})(1-q^{2r+s-n})}\,, \\
& \cB^{-}(r,s)=-\frac{q^{r+s-n-1}(1-q^{r})(1-q^{r+s-m-1})(1-q^{r+m-n-1})}{(1-q^{2r+s-n-2})(1-q^{2r+s-n-1})}\,, \\
& \cB^{0}(r,s)=-\cB^{+}(r,s)-\cB^{-}(r,s)\,,
\end{align}
\end{subequations}
and 
\begin{subequations}
\begin{align}\label{eq:diffT2}
 q^{-i}\theta_j T_{ij}(r,s)=& \sum_{\epsilon,\epsilon' \in \{0,+1,-1\}}  \cC^{\epsilon,\epsilon'}(r,s)\ T_{ij}(r+\epsilon,s+\epsilon') \,,
\end{align}
where 
\allowdisplaybreaks
\begin{align}
& \cC^{+0}(r,s)= -(1-q^{s-\ell}) \cB^+(r,s)    \,, \\
& \cC^{-0}(r,s)= -(1-q^{s-\ell}) \cB^-(r,s)    \,, \\
& \cC^{00}(r,s)= -(1-q^{s-\ell}) \cB^0(r,s)- (1-q^{s-\ell})(1-q^{s-m}) +q^{s-\ell-m-1}(1-q^{s})    \,, \\
& \cC^{-+}(r,s)=-\frac{q^{s-m}(1-q^{s-\ell})(1-q^r)(1-q^{r+m-n-1})}{1-q^{2r+s-n-1}}   \,, \\
& \cC^{0+}(r,s)=\frac{(1-q^{s-\ell})(1-q^{r+s-n-1})(1-q^{r+s-m})}{1-q^{2r+s-n-1}}   \,, \\
& \cC^{+-}(r,s)=-\frac{q^{s-\ell-m-1}(1-q^{s})(1-q^{r+m-n})}{1-q^{2r+s-n-1}}   \,, \\
& \cC^{0-}(r,s)=-\frac{q^{r+s-\ell-n-1}(1-q^{s})(1-q^{r+s-m-1})}{1-q^{2r+s-n-1}}   \,, \\
& \cC^{++}(r,s) = \cC^{--}(r,s)=0 \,.
\end{align}
\end{subequations}
\end{prop}
\begin{proof}
Equation \eqref{eq:diffT1} is directly proven by writing the eigenvalues $T_{ij}(r,s)$ as in \eqref{eq:Tij} and using the difference equation \eqref{eq:cdiffE} with $\epsilon=0$ on the polynomials $E_i(n-s,m-s;q,r)$.

To prove \eqref{eq:diffT2}, one starts similarly by using \eqref{eq:Tij} and the difference equation \eqref{eq:diffK} on the polynomials $K_j(m-i,\ell;q;s)$:
\begin{align}
 q^{-i}\theta_j T_{ij}(r,s)= &q^{i\ell} q^{s-m}(q^{m-s-i}-1)(1-q^{s-\ell}) K_j(m-i,\ell;q;s+1)E_i(n-s,m-s;q;r) \nonumber \\
 + &q^{i\ell}\left(-(q^{-i}-q^{s-m})(1-q^{s-\ell})+q^{s-\ell-m-1}(1-q^{s})\right) K_j(m-i,\ell;q;s)E_i(n-s,m-s;q;r) \nonumber \\ 
 -& q^{i\ell}q^{s-\ell-m-1}(1-q^{s}) K_j(m-i,\ell;q;s-1)E_i(n-s,m-s;q;r)   \,. \label{eq:diffT21}
\end{align}
The contiguity difference relation \eqref{eq:diffE} with $\epsilon=-$ can be called upon to rewrite the first term in the R.H.S. of \eqref{eq:diffT21} in terms of polynomials $E_i$ with shifted parameter $s \to s+1$. Similarly, the contiguity difference relation \eqref{eq:diffE} with $\epsilon=+$ allows to rewrite the third term in the R.H.S. of \eqref{eq:diffT21} in terms of polynomials $E_i$ with shifted parameter $s \to s-1$. Finally, the difference relation \eqref{eq:diffE} with $\epsilon=0$ can be used on the second term in \eqref{eq:diffT21}. Equation \eqref{eq:diffT2} is then obtained by identifying the polynomials $T_{ij}(r+\epsilon,s+\epsilon')$ using \eqref{eq:Tij}.   
\end{proof}

\subsection{Bispectral algebra} \label{ssec:bispalg}

As just shown, the polynomials $T_{ij}(r,s)$ are bispectral. It is instructive
to study the associated algebra given that for the univariate Askey--Wilson polynomials this leads to the Askey--Wilson algebra \cite{Zhe,CFGDRV} and its special cases.

Let $W$ be the vector space spanned by $\{T_{ij}\}_{(i, j)\in\cD}$. Let $\cX^\star$ and $\cY^\star$ be the matrices representing the action of the difference operators on $W$
\begin{subequations}\label{eq:XYs}
    \begin{align}
   & \cX^\star T_{ij}(x, y)=q^{-i}\theta_j T_{ij}(x, y)\,,\\
      & \cY^\star T_{ij}(x, y)=\theta_i T_{ij}(x, y)\,,
\end{align}
\end{subequations}
 and $\cX$ and $\cY$ be the matrices representing the action of the recurrence operators
 \begin{subequations}\label{eq:XY}
    \begin{align}
   & \cX T_{ij}(x, y)=\sum_{\epsilon\in \{0,+1,-1\} }  \cb^{\epsilon}(i,j)\ T_{ij+\epsilon}(x,y)\,,\\
      & \cY T_{ij}(x, y)=\sum_{\epsilon,\epsilon' \in \{0,+1,-1\}}  \cc^{\epsilon,\epsilon'}(i,j)\ T_{i+\epsilon,j+\epsilon'}(x,y) \,.
\end{align}
\end{subequations}
The bispectral algebra of $T_{ij}(r,s)$ is defined by the generators $\cX$, $\cY$, $\cX^*$ and $\cY^*$ and their commutation relations, more conveniently presented in terms of the following affine transformed operators: 
\begin{align}\label{eq:transformations}
    X= \mathcal{I}+\cX\,,\quad Y=\frac{q^{n+1}}{(q-1)^2}\big( (1+q^{-n-1})\mathcal{I}+\cX+\cY\big)\,,\quad Y^\star=\mathcal{I}+\cY^\star\,,\quad X^\star=\frac{q^{\ell+1}}{(q-1)^2}\big(\mathcal{I}+\cY^\star+\cX^\star)\,,
\end{align}
where $\mathcal{I}$ is the identity matrix. For these generators, a straightforward computation yields the following commutation relations
\begin{equation}\label{eq:bial1}
    [X,Y]=0\,,\qquad [X^\star,Y^\star]=0\,,
\end{equation}
but also 
\begin{equation}\label{eq:bial2}
    [X,Y^\star]=0\,.
\end{equation}
The relation \eqref{eq:bial2} arises from the fact that the eigenvalues of $Y$ are independent of the label $j$, while the matrix $X$ preserves the remaining index $i$. The pair $X$ and $X^\star$ satisfies the defining relations of the Askey--Wilson algebra with $Y^\star$ as a central element,  
\begin{align}\label{eq:bial3}
   & \{X^2,X^\star\}-(q+q^{-1})XX^\star X=q^{-m-1}(q+1) \mathcal{I}-q^{-m-1}(q^{\ell+1}+1)X-Y^\star X\,,
\end{align}
and
\begin{align}\label{eq:bial4}
   & \{(X^\star)^2,X\}-(q+q^{-1})X^\star X X^\star=\frac{q^{\ell-m}(q+1)}{(q-1)^2}Y^\star -q^{-m-1}(q^{\ell+1}+1)X^\star-Y^\star X^\star\,.
\end{align}
Similarly, $Y$ and $Y^\star$ also satisfy the defining relations of the Askey--Wilson algebra with $X$ as a central element, 
\begin{align}\label{eq:bial5}
   \{Y^2,Y^\star\}-(q+q^{-1})YY^\star Y=& \frac{q+1}{(q-1)^2}\big( q^{n-m-1} \mathcal{I}+(q^{m-1}+2q^n)X  -q^{n-1}(q+1)XY^\star \big) \nonumber\\
 & -(2q^{-1}+q^{n-m})Y -q^m XY \,,
\end{align}
and
\begin{align}\label{eq:bial6}
   & \{(Y^\star)^2,Y\}-(q+q^{-1})Y^\star Y Y^\star= (1+q^{-1})\mathcal{I}-q^m XY^\star -(2q^{-1}+q^{n-m})Y^\star \,.
\end{align}

\section{Bivariate $P$- and $Q$-polynomial association schemes} \label{sec:bivariatePQpoly}

In this section, we recall the definitions of bivariate $P$- and $Q$-polynomial association schemes. The bivariate case was first defined in \cite{BCPVZ}, and a multivariate extension was proposed in \cite{BKZZ}. In this paper, we will use a refinement of the definitions of \cite{BKZZ} which was considered in \cite{BCVZZ}, focusing on the case of two variables. The use of a refinement is appropriate here since the association schemes based on attenuated spaces have a structure which can be described with more precision than the general characterization of \cite{BKZZ}, which encompasses more possibilities.

We start with the definition of a special total order on the set of monomials in $\mathbb{C}[x_1, x_2]$.
\begin{defi}\label{def:mo} A monomial order $\leq$ on $\mathbb{C}[x_1, x_2]$ is a relation on the set of
monomials $x_1^{n_1}x_2^{n_2}$ satisfying:
\begin{itemize}
\item[(i)] $\leq$ is a total order;
\item[(ii)] for monomials $u$, $v$, $w$, if $u \leq v$, then $wu \leq wv$;
\item[(iii)] $\leq$ is a well-ordering, i.e. any non-empty subset of the set of monomials has a minimum element under $\leq$.
\end{itemize}
\end{defi}
Each monomial $x_1^{n_1}x_2^{n_2}$ can be associated to the tuple  $(n_1,n_2) \in \mathbb{N}^2$. We shall use the same notation $\leq$ to order tuples in $\mathbb{N}^2$. A monomial order allows to define a generalized notion of degree.
\begin{defi}\label{def:multideg} For a monomial order $\leq$ on $\mathbb{N}^2$, a bivariate polynomial $v(x_1, x_2)$ is said to be of \textit{multidegree} $(n_1,n_2) \in \mathbb{N}^2$ if $v(x_1,x_2)$ is of the form 
\begin{equation}
v(x_1,x_2) = \sum_{\substack{(m_1,m_2) \leq (n_1,n_2)}} f_{m_1,m_2} x_1^{m_1} x_2^{m_2}\,, \qquad \text{with } f_{n_1,n_2} \neq 0\,.
\end{equation}
\end{defi}

To describe more precisely the bivariate polynomial structures of association schemes, it is useful to introduce a partial order $\po$ on $\NN^2$, in addition to a fixed monomial order $\leq$, with the following properties (see \cite{BCVZZ} for more details):
\begin{itemize}
\item[(i)] for all $a,b \in \NN^2$, if $a \po b$ then  $a  \leq b$;
\item[(ii)] for all $a,b,c \in \NN^2$, if $a \po b$ then $a + c \po b + c$;
\item[(iii)] for all $a \in \NN^2$, we have $(0,0) \po a$.
\end{itemize}
The previous properties mean that the partial order $\po$ satisfies properties similar to a monomial order, and is compatible with the fixed order $\leq$ in the sense of (i). We shall assume in the following that $(\leq,\po)$ is a pair of monomial order and partial order with the above properties.  
\begin{defi}\label{def:compa}
A bivariate polynomial $v_{n_1,n_2}(x_1,x_2)$ of multidegree $(n_1,n_2) \in \NN^2$ with respect to $\leq$ is called $\po$-compatible if all monomials $x_1^{m_1}x_2^{m_2}$ appearing in $v_{n_1,n_2}(x_1,x_2)$ satisfy $(m_1,m_2) \po (n_1,n_2)$.
\end{defi} 

Let $e_1=(1,0)$ and $e_2=(0,1)$. 
Using the ideas from \cite{BCPVZ}, and based on the definition of bivariate $P$-polynomial association scheme introduced in \cite{BKZZ}, we define 
bivariate $P$-polynomial association schemes with a refined structure, which generalize the notion of $P$-polynomial association scheme.  
\begin{defi}\label{def:biPpoly} Let $\mathcal{D} \subset \mathbb{N}^2$ contain $e_1, e_2$, and let $(\leq,\po)$ be a pair of monomial order and partial order on $\mathbb{N}^2$. A commutative association scheme $\mathcal{Z}$ is called $\po$-compatible bivariate $P$-polynomial on the domain $\mathcal{D}$ with respect to $\leq$ if the following three conditions are satisfied:
\begin{enumerate}
\item[(i)] If $n = (n_1, n_2) \in \mathcal{D}$  and $0 \leq n_i' \leq  n_i$ for $i=1,2$, then $n' = (n_1', n_2') \in \mathcal{D}$;
\item[(ii)]  There exists a relabeling of the adjacency matrices of $\mathcal{Z}$:
\begin{equation}
\mathcal{Z} = \{A_n \ | \ n \in \mathcal{D}\}
\end{equation}
such that for all $n \in \mathcal{D}$ we have
\begin{equation}\label{eq:polyA}
A_n = v_n(A_{10}, A_{01})
\end{equation}
where $v_n(x_1,x_2)$ is a $\po$-compatible bivariate polynomial of multidegree $n$ and all monomials $x_1^{a_1}x_2^{a_2}$ in $v_n(x_1,x_2)$ satisfy $(a_1,a_2) \in \mathcal{D}$;

\item[(iii)] For $i = 1,2$ and $a = (a_1, a_2) \in \mathcal{D}$, the product $A_{e_i}  A_{10}^{a_1}A_{01}^{a_2}$ is a
linear combination of
\begin{equation}
\{ A_{10}^{b_1}A_{01}^{b_2}\ | \ b = (b_1,b_2) \in \mathcal{D}, \ b \po a + e_i  \}\,.
\end{equation}
\end{enumerate}
\end{defi}
Such bivariate $P$-polynomial association schemes can be characterized in terms of conditions on their intersection numbers, as shown in \cite{BCPVZ,BKZZ,BCVZZ}.
\begin{prop}\label{prop:biPpoly2} Let $\mathcal{D}\subset \mathbb{N}^2$ such that $e_1, e_2\in \mathcal{D}$ and $\mathcal{Z} = \{A_n \ | \ n \in \mathcal{D}\}$ be a commutative association scheme. Then the following two statements are
equivalent:
\begin{enumerate}
\item[(i)] $\mathcal{Z}$ is a $\po$-compatible bivariate $P$-polynomial association scheme on $\mathcal{D}$ with respect to $\leq$;
\item[(ii)] The condition (i) of Definition \ref{def:biPpoly} holds for $\mathcal{D}$ and the intersection numbers satisfy,  
for each $i = 1,2$ and each $a \in \mathcal{D}$, $p_{e_i, a}^b \neq 0$ for $b \in \mathcal{D}$ implies $b \po a + e_i$. Moreover, if $a + e_i \in \mathcal{D}$, then $p_{e_i, a}^{a + e_i} \neq 0$ holds.
\end{enumerate}
\end{prop}
An analogous definition exists for (refined) bivariate $Q$-polynomial association schemes as well as a characterization of such schemes in terms of Krein parameters.  
\begin{defi}\label{def:biQpoly} Let $\mathcal{D}^\star \subset \mathbb{N}^2$ contain $e_1, e_2$, and let $(\leq,\po)$ be a pair of monomial order and partial order on $\mathbb{N}^2$. A commutative association scheme $\mathcal{Z}$ is called $\po$-compatible bivariate $Q$-polynomial on the domain $\mathcal{D}^\star$ with respect to $\leq$ if the following three conditions are satisfied:
\begin{enumerate}
\item[(i)] If $n = (n_1, n_2) \in \mathcal{D}^\star$  and $0 \leq n_i' \leq  n_i$ for $i=1,2$, then $n' = (n_1', n_2') \in \mathcal{D}^\star$;
\item[(ii)]  There exists a relabeling of the primitive idempotent matrices of $\mathcal{Z}$:
\begin{equation}
\{E_n \ | \ n \in \mathcal{D}^\star\}
\end{equation}
such that for all $n \in \mathcal{D}^\star$ we have
\begin{equation}\label{eq:polyE}
|X|E_n = v_n^\star(|X|E_{10},|X|E_{01}) \qquad \text{(under Schur product)}
\end{equation}
where $X$ is the vertex set of the scheme, and $v_n^\star(x_1,x_2)$ is a $\po$-compatible bivariate polynomial of multidegree $n$ and all monomials $x_1^{a_1}x_2^{a_2}$ in $v_n^\star(x_1,x_2)$ satisfy $(a_1,a_2) \in \mathcal{D}^\star$;

\item[(iii)] For $i = 1,2$ and $a = (a_1, a_2) \in \mathcal{D}$, the product $E_{e_i} \circ  E_{10}^{\circ a_1}\circ E_{01}^{\circ a_2}$ is a
linear combination of
\begin{equation}
\{ E_{10}^{\circ b_1}\circ E_{01}^{\circ b_2}\ | \ b = (b_1,b_2) \in \mathcal{D}^\star, \ b \po a + e_i  \} \,.
\end{equation}
\end{enumerate}
\end{defi}

\begin{prop}\label{prop:biQpoly2} Let $\mathcal{D}^\star\subset \mathbb{N}^2$ such that $e_1, e_2\in \mathcal{D}^\star$ and $\mathcal{Z}$ be a commutative association scheme with primitive idempotents $\{E_n \ | \ n \in \mathcal{D}^\star\}$. Then the following two statements are
equivalent:
\begin{enumerate}
\item[(i)] $\mathcal{Z}$ is a $\po$-compatible bivariate $Q$-polynomial association scheme on $\mathcal{D}^\star$ with respect to 
$\leq$;
\item[(ii)] The condition (i) of Definition \ref{def:biQpoly} holds for $\mathcal{D}^\star$ and the Krein numbers satisfy,  
for each $i = 1,2$ and each $a \in \mathcal{D}^\star$, $q_{e_i, a}^b \neq 0$ for $b \in \mathcal{D}$ implies $b \po a + e_i$. Moreover, if $a + e_i \in \mathcal{D}^\star$, then $q_{e_i, a}^{a + e_i} \neq 0$ holds.
\end{enumerate}
\end{prop}

\begin{rem}
Let us emphasize that for a symmetric association scheme such as $\A$, $p_{a,b}^c=0$ if and only if $p_{a,c}^b=0$. Therefore, in item (ii) of Proposition \ref{prop:biPpoly2} we also have that for $a,b\in\cD$, $p_{e_i,a}^b\neq 0$ implies $a \po b+e_i$, and if $a,a-e_i \in \cD$ then $p_{e_i,a}^{a-e_i} \neq 0$. Similar results hold for $q_{a,b}^c$ in (ii) of Proposition \ref{prop:biQpoly2}.
\end{rem}

\begin{rem}
A $\po$-compatible bivariate $P$- or $Q$-polynomial association scheme with respect to $\leq$ is bivariate $P$- or $Q$-polynomial with respect to $\leq$ in the sense of \cite{BKZZ}, as $a \po b$ implies $a \leq b$ for any $a,b \in \NN^2$. 
\end{rem}

In the following, the monomial order that we will use is the \textit{deg-lex} order, defined on $\NN^2$ by
\begin{equation}\label{eq:deglex}
 (m_1,m_2) \leq (n_1,n_2) \iff \left\{
	\begin{array}{ll}
		 m_1+m_2 < n_1+n_2\,, \\
		 \text{or} \\
		m_1+m_2 = n_1+n_2 \ \text{and} \ m_1\leq n_1\,.
	\end{array}
\right. 
\end{equation} 
In terms of monomials $x_1^{n_1}x_2^{n_2}$, the order \eqref{eq:deglex} satisfies $x_2 < x_1$. 
An alternative deg-lex order that we will also need, denoted here by $\leq'$, is obtained by exchanging the roles of the two coordinates,
\begin{equation}\label{eq:deglex1}
 (m_1,m_2) \leq' (n_1,n_2) \iff \left\{
	\begin{array}{ll}
		 m_1+m_2 < n_1+n_2\,, \\
		 \text{or} \\
		m_1+m_2 = n_1+n_2 \ \text{and} \ m_2\leq n_2\,.
	\end{array}
\right. 
\end{equation} 
Note that $x_1 <' x_2$.

Furthermore, we will use a partial order, denoted $\poo$, that depends on two parameters $\alpha,\beta$ satisfying $0\leq \alpha,\beta \leq 1$ and $\alpha\beta\neq1$, and that is defined as follows on $\NN^2$:
\begin{equation}\label{eq:po}
(m_1,m_2) \preceq_{(\alpha, \beta)} (n_1,n_2) \iff \left\{
	\begin{array}{ll}
		 m_1+ \alpha m_2 \leq n_1 + \alpha n_2 \\
		 \text{and} \\
		\beta m_1 + m_2 \leq \beta n_1 + n_2\,.
	\end{array}
\right.
\end{equation}
The partial order defined by \eqref{eq:po} was used in \cite{BCPVZ} to study the bivariate $P$-polynomial structures of several examples of higher rank association schemes, and it is further discussed in \cite{BCVZZ}. In order for $\poo$ to be compatible with the deg-lex order $\leq$ defined in \eqref{eq:deglex}, one must restrict $0\leq \alpha<1$ and $0\leq \beta \leq 1$, and to be compatible with the deg-lex order $\leq'$ defined in \eqref{eq:deglex1} and used in \cite{BCPVZ}, one must restrict $0\leq \alpha \leq 1$ and $0\leq \beta<1$.

Figures \ref{fig:intnumb} and \ref{fig:Kreinparam} illustrate which intersection numbers or Krein parameters can be non-zero, according to Propositions \ref{prop:biPpoly2} and \ref{prop:biQpoly2}; this will be helpful later when checking some specific examples of $\poo$-compatible bivariate $P$- or $Q$-polynomial structures. 
In addition to these possibly non-vanishing parameters, the general bivariate $P$- or $Q$-polynomial association schemes defined in \cite{BKZZ} (without the use of a compatible partial order) allow more possibly non-vanishing intersection numbers or Krein parameters, represented by the smaller dots in the figure. This illustrates that the compatibility of a bivariate $P$- or $Q$-polynomial association scheme with a partial order carries important structural information of the scheme. 
\begin{figure}[hbtp]
\begin{center}
\begin{subfigure}[b]{0.45\textwidth}
\centering
\begin{tikzpicture}[scale=0.45]
\draw[->] (0,0)--(11,0);\draw[->] (0,0)--(0,11);
\draw [fill] (4,5) circle (\brad);
\draw [fill] (4,6) circle (\brad);
\draw [fill] (4,4) circle (\brad);
\draw [fill] (3,5) circle (\brad);
\draw [fill] (3,6) circle (\brad);
\draw [fill] (3,7) circle (\brad);
\draw [fill] (5,5) circle (\brad);
\draw [fill] (5,4) circle (\brad);
\draw [fill] (5,3) circle (\brad);

\draw [fill] (6,3) circle (\srad);
\draw [fill] (7,2) circle (\srad);
\draw [fill] (8,1) circle (\srad);
\draw [fill] (9,0) circle (\srad);

\draw [fill] (2,7) circle (\srad);
\draw [fill] (1,8) circle (\srad);
\draw [fill] (0,9) circle (\srad);

\draw [fill] (6,2) circle (\srad);
\draw [fill] (7,1) circle (\srad);
\draw [fill] (8,0) circle (\srad);

\draw [fill] (2,8) circle (\srad);
\draw [fill] (1,9) circle (\srad);
\draw [fill] (0,10) circle (\srad);

\draw (0,11) node[above] {$n_2$};
\draw (11,0) node[above] {$n_1$};
\draw (4,0) node[below] {$m_1$};
\draw (0,5) node[left] {$m_2$};
\draw[dashed,thin] (4,0)--(4,5) --(0,5);
\end{tikzpicture}
\caption{$p_{10,m}^n$, $(\leq,\po_{(0,1)})$}
\label{fig:recp10}
\end{subfigure}%
\begin{subfigure}[b]{0.45\textwidth}
\centering
\begin{tikzpicture}[scale=0.45]
\draw[->] (0,0)--(11,0);\draw[->] (0,0)--(0,11);
\draw [fill] (4,5) circle (\brad);
\draw [fill] (4,6) circle (\brad);
\draw [fill] (4,4) circle (\brad);

\draw [fill] (5,4) circle (\srad);
\draw [fill] (6,3) circle (\srad);
\draw [fill] (7,2) circle (\srad);
\draw [fill] (8,1) circle (\srad);
\draw [fill] (9,0) circle (\srad);

\draw [fill] (3,6) circle (\srad);
\draw [fill] (2,7) circle (\srad);
\draw [fill] (1,8) circle (\srad);
\draw [fill] (0,9) circle (\srad);

\draw [fill] (5,3) circle (\srad);
\draw [fill] (6,2) circle (\srad);
\draw [fill] (7,1) circle (\srad);
\draw [fill] (8,0) circle (\srad);

\draw [fill] (2,8) circle (\srad);
\draw [fill] (3,7) circle (\srad);
\draw [fill] (1,9) circle (\srad);
\draw [fill] (0,10) circle (\srad);

\draw (0,11) node[above] {$n_2$};
\draw (11,0) node[above] {$n_1$};
\draw (4,0) node[below] {$m_1$};
\draw (0,5) node[left] {$m_2$};
\draw[dashed,thin] (4,0)--(4,5) --(0,5);
\end{tikzpicture}
\caption{$p_{01,m}^n$, $(\leq,\po_{(0,1)})$}
\label{fig:recp01}
\end{subfigure}%
\end{center}
\caption{Examples of possibly non-vanishing intersection numbers. For a fixed point $m=(m_1,m_2)$, the coordinates $n=(n_1,n_2)$ of the larger dots in the graphs represent when (a) $p_{10,m}^{n}$ and (b) $p_{01,m}^{n}$ may be non-zero for a $\po_{(0,1)}$-compatible bivariate $P$-polynomial association scheme with respect to $\leq$. The smaller dots must be added without the $\po_{(0,1)}$-compatibility. 
Note that all these dots lie on three lines: $n_1+n_2=m_1+m_2$ and $n_1+n_2=m_1+m_2 \pm 1$ with the upper line truncated below, and the lower line truncated above.}
\label{fig:intnumb}  
\end{figure}
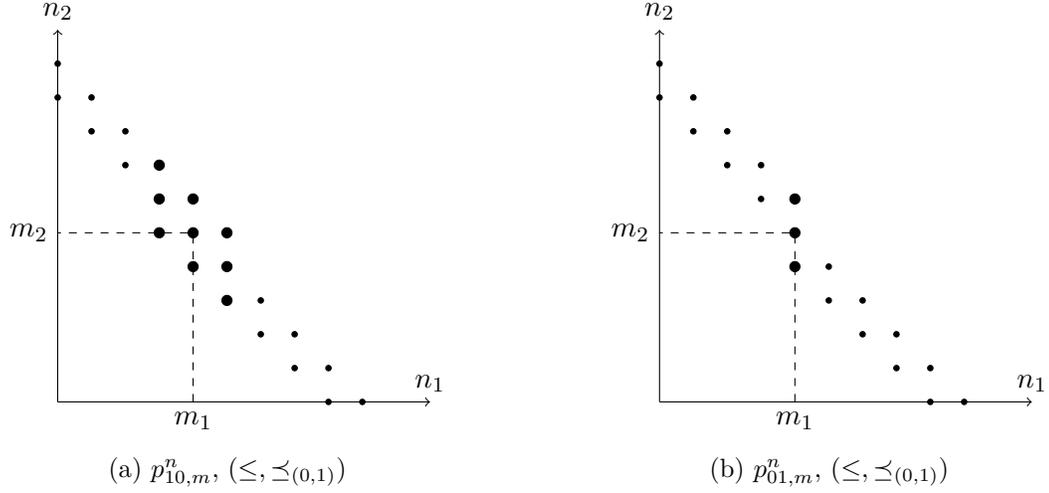
\begin{figure}[hbtp]
\begin{center}
\begin{subfigure}[b]{0.45\textwidth}
\centering
\begin{tikzpicture}[scale=0.45]
\draw[->] (0,0)--(11,0);\draw[->] (0,0)--(0,11);
\draw [fill] (4,5) circle (\brad);
\draw [fill] (3,5) circle (\brad);
\draw [fill] (5,5) circle (\brad);

\draw [fill] (5,4) circle (\srad);
\draw [fill] (6,3) circle (\srad);
\draw [fill] (7,2) circle (\srad);
\draw [fill] (8,1) circle (\srad);
\draw [fill] (9,0) circle (\srad);

\draw [fill] (3,6) circle (\srad);
\draw [fill] (2,7) circle (\srad);
\draw [fill] (1,8) circle (\srad);
\draw [fill] (0,9) circle (\srad);

\draw [fill] (3,5) circle (\srad);
\draw [fill] (2,6) circle (\srad);
\draw [fill] (1,7) circle (\srad);
\draw [fill] (0,8) circle (\srad);

\draw [fill] (6,4) circle (\srad);
\draw [fill] (7,3) circle (\srad);
\draw [fill] (8,2) circle (\srad);
\draw [fill] (9,1) circle (\srad);
\draw [fill] (10,0) circle (\srad);

\draw (0,11) node[above] {$n_2$};
\draw (11,0) node[above] {$n_1$};
\draw (4,0) node[below] {$m_1$};
\draw (0,5) node[left] {$m_2$};
\draw[dashed,thin] (4,0)--(4,5) --(0,5);
\end{tikzpicture}
\caption{$q_{10,m}^n$, $(\leq',\po_{(1,0)})$}
\label{fig:recq10}
\end{subfigure}%
\begin{subfigure}[b]{0.45\textwidth}
\centering
\begin{tikzpicture}[scale=0.45]
\draw[->] (0,0)--(11,0);\draw[->] (0,0)--(0,11);
\draw [fill] (4,5) circle (\brad);
\draw [fill] (5,5) circle (\brad);
\draw [fill] (3,5) circle (\brad);
\draw [fill] (2,6) circle (\brad);
\draw [fill] (3,6) circle (\brad);
\draw [fill] (4,6) circle (\brad);
\draw [fill] (4,4) circle (\brad);
\draw [fill] (5,4) circle (\brad);
\draw [fill] (6,4) circle (\brad);

\draw [fill] (6,3) circle (\srad);
\draw [fill] (7,2) circle (\srad);
\draw [fill] (8,1) circle (\srad);
\draw [fill] (9,0) circle (\srad);

\draw [fill] (2,7) circle (\srad);
\draw [fill] (1,8) circle (\srad);
\draw [fill] (0,9) circle (\srad);

\draw [fill] (1,7) circle (\srad);
\draw [fill] (0,8) circle (\srad);

\draw [fill] (8,2) circle (\srad);
\draw [fill] (7,3) circle (\srad);
\draw [fill] (9,1) circle (\srad);
\draw [fill] (10,0) circle (\srad);

\draw (0,11) node[above] {$n_2$};
\draw (11,0) node[above] {$n_1$};
\draw (4,0) node[below] {$m_1$};
\draw (0,5) node[left] {$m_2$};
\draw[dashed,thin] (4,0)--(4,5) --(0,5);
\end{tikzpicture}
\caption{$q_{01,m}^n$, $(\leq',\po_{(1,0)})$}
\label{fig:recq01}
\end{subfigure}%
\end{center}
\caption{Examples of possibly non-vanishing Krein parameters. For a fixed point $m=(m_1,m_2)$, the coordinates $n=(n_1,n_2)$ of the larger dots in the graphs represent when (a) $q_{10,m}^{n}$ and (b) $q_{01,m}^{n}$ may be non-zero for a $\po_{(1,0)}$-compatible bivariate $Q$-polynomial association scheme with respect to $\leq'$. The smaller dots must be added without the $\po_{(1,0)}$-compatibility. Note that all these dots lie on three lines: $n_1+n_2=m_1+m_2$ and $n_1+n_2=m_1+m_2 \pm 1$ with the upper line truncated above, and the lower line truncated below.}
\label{fig:Kreinparam}  
\end{figure}
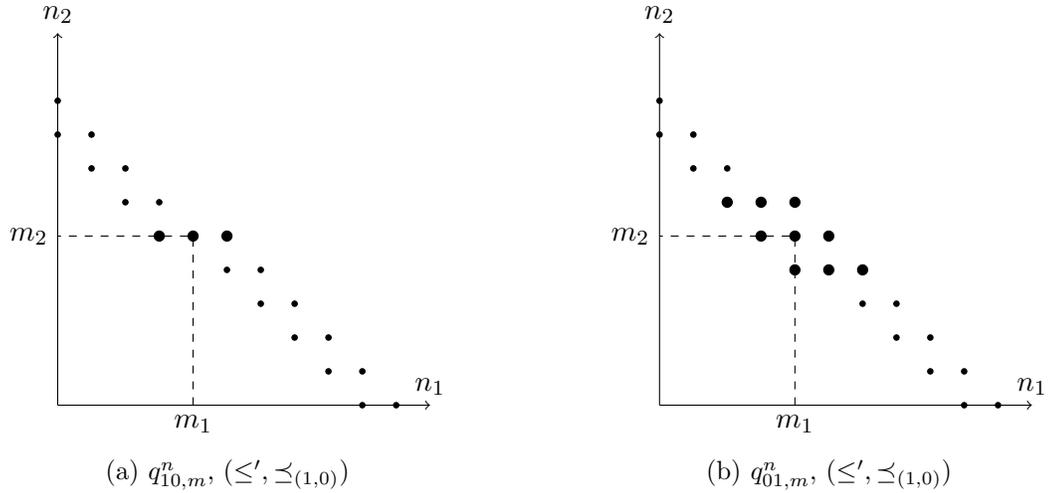

\section{Bivariate $P$-polynomial structure of $\A$} \label{sec:biPpoly}
We are now ready to identify the bivariate $P$-polynomial structure of association schemes based on attenuated spaces $\A$, using the previous results and definitions. 

From equation~\eqref{eq:Tij}, one can obtain all the eigenvalues $T_{10}(r,s)$ and $T_{01}(r,s)$ of the matrices $A_{10}$ and $A_{01}$, respectively, with $(r,s)\in \cD^\star$. These eigenvalues are written as follows in terms of
$\lambda(s)=q^{-s}-1$ and $\Lambda(r,s)=(q^{-r}-1)(1-q^{r+s-n-1})$: 
\begin{align}
&T_{10}(r,s)
=\frac{q^{\ell+n+1}}{(1-q)^2}\left((1-q^{-m})(1-q^{m-n})+(1-q^{m-n})\lambda(s)+ q^{-s}\Lambda(r,s)\right)\,, \label{eq:T10}\\
&T_{01}(r,s)
= \frac{-q^{\ell+m}}{1-q} \left((1-q^{-\ell})(1-q^{-m})+\lambda(s)\right)\,. \label{eq:T01}
\end{align}
Using the recurrence relations in Proposition \ref{prop:recT}, we derive the following 
formulas for the product of the eigenvalues $T_{10}(r,s)$ and $T_{01}(r,s)$ with $T_{ij}(r,s)$.

\begin{prop}\label{prop:recTPpoly}
The eigenvalues $T_{ij}(r,s)$ of association schemes based on attenuated spaces $\A$ satisfy the following relations for all $(i,j) \in \cD$ and $(r,s) \in \cD^\star$:
\begin{subequations}
\begin{align}\label{eq:recT10}
 T_{01}(r,s) T_{ij}(r,s)=& \sum_{\epsilon\in \{0,+1,-1\} }  \tcb^{\epsilon}(i,j)\ T_{i,j+\epsilon}(r,s) \,,
\end{align}
where
\begin{align}
&\tcb^{+}(i,j)=\frac{q^{i+j}(1-q^{j+1})}{1-q}   \,, \\
&\tcb^{-}(i,j)=
\frac{(q^\ell-q^{j-1})(q^{i+j-1}-q^{m})}{1-q}
\,, \\
&\tcb^{0}(i,j)=
\frac{(q^\ell-1)(1-q^m)-q^{i+j-1}(1-q^j)-(q^\ell-q^{j})(q^{i+j}-q^m)}{1-q}\,,
\end{align}
\end{subequations}
and 
\begin{subequations}
\begin{align}\label{eq:recT01}
 T_{10}(r,s) T_{i,j}(r,s)=& \sum_{\epsilon,\epsilon' \in \{0,+1,-1\}}  \tcc^{\epsilon,\epsilon'}(i,j)\ T_{i+\epsilon,j+\epsilon'}(r,s) \,,
\end{align}
where
\allowdisplaybreaks
\begin{align}
&\tcc^{+0}(i,j)=
\frac{q^j(1-q^{i+1})^2}{(1-q)^2}
\,, \\
&\tcc^{+-}(i,j)= \frac{(q^{\ell}-q^{j-1})(1-q^{i+1})^2}{(1-q)^2}   \,, \\
&\tcc^{0+}(i,j)
=\frac{q^{i+j+1}(1-q^i)(1-q^{j+1})} {(1-q)^2}   \,, \\
& \tcc^{00}(i,j)=\frac{q^{\ell+1}\left(1-q^i\right)}{(1-q)^2} \left(1-q^{n-m}+q^i-q^m-q^{-1}(1-q^i)-q^m\cb^0(i,j)\right) \,, \\
&\tcc^{0-}(i,j)=\frac{(q^{\ell}-q^{j-1})(1-q^i)(q^{i+j}-q^{m+1})}{(1-q)^2}   \,, \\
&\tcc^{-+}(i,j)=\frac{q^{\ell+i}(1-q^{j+1})(q^{i-1}-q^{n-m})}{(1-q)^2}   \,, \\
& \tcc^{-0}(i,j)=\frac{q^{\ell}(q^{i+j}-q^{m+1})(q^{i-1}-q^{n-m})}{(1-q)^2}   \,, \\
& \tcc^{++}(i,j) = \tcc^{--}(i,j)=0 \,.
\end{align}
\end{subequations}
\end{prop}
\begin{proof}
The two product formulas follow straightforwardly from equations \eqref{eq:T10}, \eqref{eq:T01} and Proposition~\ref{prop:recT}. 
\end{proof}

\paragraph{Intersection numbers.}
By substituting the first relation of \eqref{eq:eigenvalues} in \eqref{eq:Amnij}, and using the property $E_{ij} E_{mn} = \delta_{im}\delta_{jn} E_{ij}$, one can identify the coefficients in the product formulas of the eigenvalues $T_{ij}(r,s)$ with the intersection numbers of $\A$, \textit{i.e.} 
\begin{equation}
T_{mn}(r,s)T_{ij}(r,s) = \sum_{(a,b) \in \cD} p_{mn,ij}^{ab} T_{ab}(r,s)\,.
\end{equation}
Therefore, one obtains the following from Proposition \ref{prop:recTPpoly}.
\begin{coro}\label{coro:intnumb}
The non-vanishing intersection numbers $p_{10,ij}^{ab}$ and $p_{01,ij}^{ab}$ of the association schemes based on attenuated spaces $\A$ are given by
\begin{equation}\label{eq:intnumb}
p_{10,ij}^{i+\epsilon,j+\epsilon'} = \tcc^{\epsilon,\epsilon'}(i,j) \,, \qquad  p_{01,ij}^{i,j+\epsilon} = \tcb^{\epsilon}(i,j)\, , \qquad \text{for } \epsilon,\epsilon' \in \{0,+1,-1\}\,,
\end{equation}
where $\tcb^{\epsilon}(i,j)$ and $\tcc^{\epsilon,\epsilon'}(i,j)$ are the coefficients given in Proposition \ref{prop:recTPpoly}. 
\end{coro}
Note that the intersection numbers of the association schemes based on attenuated spaces were obtained in \cite{BCPVZ} by a direct combinatorial counting argument. The above is a different approach making use of the recurrence, contiguity, and difference relations of the $q$-Krawtchouk and the dual $q$-Hahn polynomials.  

\begin{coro}
    The association scheme based on attenuated spaces $\A$ is $\preceq_{(0,1)}$-compatible bivariate $P$-polynomial on the domain $\cD$ with repsect to $\leq$ (as in \eqref{eq:deglex}). 
\end{coro}
\begin{proof}
    The domain $\cD$ of the scheme as defined in \eqref{eq:domain} satisfies condition (i) of Definition~\ref{def:biPpoly}.
    Recall that the partial order $\preceq_{(0,1)}$ is compatible with the monomial order $\leq$. Using equations \eqref{eq:intnumb}, it is easy to check that for $k=1,2$, all pairs $(a,b)$ such that $p_{e_k,ij}^{ab}\neq 0$ satisfy $(a,b) \preceq_{(0,1)} e_k+(i,j)$ (see also Figure \ref{fig:intnumb}), 
    and if $e_k+(i,j)\in \cD$ then $p_{e_k,ij}^{e_k+(i,j)}\neq 0$. By Proposition~\ref{prop:biPpoly2}, the result follows. 
\end{proof}
The bivariate polynomials $v_{n}$ of Definition \ref{def:biPpoly} 
satisfy 
\begin{equation}
T_{ij}(r,s)=v_{ij}(T_{10}(r,s),T_{01}(r,s))\,, 
\end{equation}
and hence 
\begin{align*}
    &xv_{ij}(x,y)=\sum_{\epsilon,\epsilon' \in \{0,+1,-1\}}  \tcc^{\epsilon,\epsilon'}(i,j)\ v_{i+\epsilon,j+\epsilon'}(x,y)\, , \\
    &yv_{ij}(x,y)=\sum_{\epsilon\in \{0,+1,-1\} }  \tcb^{\epsilon}(i,j)\ v_{i,j+\epsilon}(x,y)\,. 
\end{align*}
All the bivariate polynomials $v_{ij}(x,y)$ can be deduced from the initial conditions $v_{00}=1$, $v_{10}=x$ and $v_{01}=y$ and the above recurrence relations. 
\begin{rem}
A bivariate $P$-polynomial structure of $\A$ was obtained in \cite{BCPVZ} for restricted parameter $\ell$, because of a restriction on the admissible domains $\cD$ in the definition of bivariate $P$-polynomial association scheme proposed in \cite{BCPVZ}. The restriction is removed in \cite{BKZZ}, where the definition of multivariate $P$-polynomial association scheme differs. For a comparison between the two definitions and discussion of their implications, see \cite{BCVZZ}. 
\end{rem}

\section{Bivariate $Q$-polynomial structure of $\A$} \label{sec:biQpoly}
We now proceed similarly to identify a bivariate $Q$-polynomial structure for the association scheme $\A$.

The dual eigenvalues $U_{10}(i,j)$ and $U_{01}(i,j)$ respectively associated to the idempotents $E_{10}$ and $E_{01}$, with $(i,j)\in \cD$, are computed from equation \eqref{eq:Urs} and written in terms of $\theta_i=q^{-i}-1$ as follows: 
\begin{align}
&U_{10}(i,j) =\frac{q-q^n}{1-q}\left(1+ \frac{(1-q^{-n})\theta_i}{(1-q^{-m})(1-q^{m-n})}\right)\,, \label{eq:U10} \\
&U_{01}(i,j)=\frac{(q^\ell-1)(1-q^n)}{1-q}\left(1+\frac{\theta_i}{(1-q^{-m})}+ \frac{q^{-i}\theta_j}{(1-q^{-m})(1-q^{-\ell})}\right)\,. \label{eq:U01}
\end{align}

\begin{prop}\label{prop:recUQpoly}
The dual eigenvalues $U_{rs}(ij)$ of association schemes based on attenuated spaces $\A$ satisfy the following relations for all $(i,j) \in \cD$ and $(r,s) \in \cD^\star$:
\begin{subequations}
\begin{align}\label{eq:diffU1}
 U_{10}(i,j) U_{rs}(i,j)=& \sum_{\epsilon\in \{0,+1,-1\} }  \tcB^{\epsilon}(r,s)\ U_{r+\epsilon,s}(i,j) \,,
\end{align}
where
\begin{align}
&\tcB^{+}(r,s)=\frac{q^{r+s}(1-q^{-n})(1-q^{-n+1})(1-q^{r+1})(1-q^{r+m-n})(1-q^{r+s-m})}{(1-q)(1-q^{-m})(1-q^{m-n})(1-q^{2r+s-n})(1-q^{2r+s-n+1})}   \,, \\
&\tcB^{-}(r,s)=\frac{q (1-q^{n-1}) (1-q^{-n}) (1-q^{r+m-n-1}) (1-q^{r+s-m-1}) (1-q^{r+s-n-2})}{(1-q) (1-q^{-m}) (1-q^{m-n}) (1-q^{2r+s-n-2})(1-q^{2r+s-n-3})}\,, \\
&\tcB^{0}(r,s)=\frac{q(1-q^{n-1})}{1-q}+\frac{q(1-q^{n-1})(1-q^{-n})}{(1-q)(1-q^{-m}) (1-q^{m-n})}\cB^0(r,s)\,,
\end{align}
\end{subequations}
and 
\begin{align}\label{eq:diffU2}
 U_{01}(i,j) U_{rs}(i,j)=& \sum_{\epsilon,\epsilon' \in \{0,+1,-1\}}  \tcC^{\epsilon,\epsilon'}(r,s)\ U_{r+\epsilon,s+\epsilon'}(i,j) \,,
\end{align}
where
\begin{subequations}
\allowdisplaybreaks
\begin{align}
&\tcC^{+0}(r,s)=-\frac{q^{r+s}(1-q^{-n})(1-q^{s})(1-q^{r+1})(1-q^{r+m-n})(1-q^{r+s-m})}{(1-q)(1-q^{-m})(1-q^{2r+s-n})(1-q^{2r+s-n+1})}\,, \\
&\tcC^{-0}(r,s)= -\frac{(1-q^{n})(1-q^{s})(1-q^{r+s-n-2})(1-q^{r+s-m-1})(1-q^{r+m-n-1})}{(1-q)(1-q^{-m})(1-q^{2r+s-n-2})(1-q^{2r+s-n-3})}   \,, \\
&\tcC^{00}(r,s)
=\frac{(1-q^n)(1-q^{s})}{(1-q)(1-q^m)}\left(q^{\ell}+q^m-q^{s}-q^{s-1}-1+q^{m}\cB^0(r,s)\right)     \,, \\
&\tcC^{-+}(r,s)=-\frac{q^{s-m}(1-q^{n})(1-q^{s+1})(1-q^{r+m-n-1})}{(1-q)(1-q^{-m})(1-q^{2r+s-n-2})}   \,, \\
&\tcC^{0+}(r,s)=\frac{q^{r+s}(1-q^{-n})(1-q^{s+1})(1-q^{r+s-m})}{(1-q)(1-q^{-m})(1-q^{2r+s-n})}   \,, \\
&\tcC^{+-}(r,s)=-\frac{q^{s-m-1}(q^{\ell}-q^{s-1})(1-q^{r+m-n})(1-q^{n})(1-q^{r+1})}{(1-q)(1-q^{-m})(1-q^{2r+s-n})}   \,, \\
& \tcC^{0-}(r,s)=\frac{(q^{\ell}-q^{s-1})(1-q^{n})(1-q^{r+s-m-1})(1-q^{r+s-n-2})}{(1-q)(1-q^{-m})(1-q^{2r+s-n-2})}   \,, \\
& \tcC^{++}(r,s) = \tcC^{--}(r,s)=0 \,.
\end{align}
\end{subequations}
\end{prop}
\begin{proof}
Using the Wilson duality, one can rewrite the difference equations \eqref{eq:diffT1} and \eqref{eq:diffT2} of the polynomials $T_{ij}(r,s)$ as recurrence relations for $U_{rs}(i,j)$ of the following form:
\begin{align}\label{eq:gendiffU}
 \nu(i,j) U_{rs}(i,j)=& \sum_{\epsilon,\epsilon'\in \{0,+1,-1\} }  \cF^{\epsilon,\epsilon'}(r,s) \frac{U_{rs}(0,0)}{U_{r+\epsilon,s+\epsilon'}(0,0)}\ U_{r+\epsilon,s+\epsilon'}(i,j) \,,
\end{align}
where $\nu(i,j)$ is either equal to $\theta_i$ or to $q^{-i}\theta_j$, and the coefficients $\cF^{\epsilon, \epsilon'}$ refer respectively to either coefficients $\cB^{\epsilon}$ or $\cC^{\epsilon, \epsilon'}$ from Proposition \ref{prop:diffT}. The dual eigenvalues $U_{10}(i,j)$ and $U_{01}(i,j)$ are given explicitly in terms of $\theta_i$ and $q^{-i}\theta_j$ in \eqref{eq:U10} and \eqref{eq:U01}. Therefore, $U_{10}(i,j)U_{rs}(i,j)$ and $U_{01}(i,j)U_{rs}(i,j)$ can be computed using both recurrence relations of the form \eqref{eq:gendiffU}. The coefficients $\tcB$ and $\tcC$ of this proposition are obtained by modifying the coefficients $\cB$ and $\cC$ as indicated in \eqref{eq:gendiffU}, and collecting similar terms together. The following intermediate results are useful for this computation and are obtained from \eqref{eq:Urs}:
\begin{align}
&U_{rs}(0,0) =q^{\ell s}(q^{-\ell};q)_s\qbinom{n}{s}\qbinom{n-s}{r}\frac{1-q^{2r+s-n-1}}{1-q^{r+s-n-1}}\,,
\end{align}
which leads to
\begin{align}
&\frac{U_{rs}(0,0)}{U_{r+1,s}(0,0)}=-\frac{q^{r+s-n}(1-q^{r+1})(1-q^{2r+s-n-1})}{(1-q^{r+s-n-1})(1-q^{2r+s-n+1})} \,,\\
&\frac{U_{rs}(0,0)}{U_{r-1,s+1}(0,0)}=\frac{q^{-\ell}(1-q^{s+1})(1-q^{2r+s-n-1})}{(1-q^{s-\ell})(1-q^{r})(1-q^{2r+s-n-2})} \,,\\
&\frac{U_{rs}(0,0)}{U_{r,s+1}(0,0)}=-\frac{q^{r+s-n-\ell}(1-q^{s+1})(1-q^{2r+s-n-1})}{(1-q^{s-\ell})(1-q^{r+s-n-1})(1-q^{2r+s-n})}\,.
\end{align}
\end{proof}

\paragraph{Krein parameters.}
By substituting the second relation of \eqref{eq:eigenvalues} in \eqref{eq:Emnrs}, and using the property $A_{ij}\circ A_{mn} = \delta_{im}\delta_{jn} A_{ij}$, one obtains a correspondence between the coefficients in product formulas for the dual eigenvalues and the Krein parameters of the scheme,
\begin{equation}
U_{mn}(ij)U_{rs}(ij) = \sum_{(a,b) \in \cD^\star} q_{mn,rs}^{ab} U_{ab}(ij) \, .
\end{equation}
Therefore, the following result regarding the non-vanishing Krein parameters can be derived from the product formulas of Proposition \ref{prop:recUQpoly}.
\begin{coro}
The non-vanishing Krein parameters $q_{10,rs}^{ab}$ and $q_{01,rs}^{ab}$ of the association schemes based on attenuated spaces $\A$ are given by
\begin{equation} \label{eq:KreinParam}
q_{10,rs}^{r+\epsilon,s} = \tcB^{\epsilon}(r,s)\, , \qquad q_{01,rs}^{r+\epsilon,s+\epsilon'} = \tcC^{\epsilon,\epsilon'}(r,s) \,, \qquad \text{for } \epsilon,\epsilon' \in \{0,+1,-1\}\,,
\end{equation}
where $\tcB^{\epsilon}(r,s)$ and $\tcC^{\epsilon,\epsilon'}(r,s)$ are the coefficients given in Proposition \ref{prop:recUQpoly}.
\end{coro}
Note that the Krein parameters of the association schemes $\A$ have been obtained in \cite{BKZZ2} with a different approach. Because of different choices of notations, the Krein parameters denoted $q_{mn,rs}^{r+\epsilon,s+\epsilon'}$ in \cite{BKZZ2} correspond to those labeled $q_{nm,sr}^{s+\epsilon',r+\epsilon}$ in the present paper.
\begin{coro}
    The association scheme based on attenuated spaces $\A$ is $\preceq_{(1,0)}$-compatible bivariate $Q$-polynomial on the domain $\cD^\star=\cD$ with respect to $\leq'$ (as in \eqref{eq:deglex1}). 
\end{coro}
\begin{proof}
    The domain $\cD^\star=\cD$ as defined in \eqref{eq:domain} satisfies condition (i) of Definition~\ref{def:biQpoly}.
    Note that the partial order $\preceq_{(1,0)}$ is compatible with the monomial order $\leq'$. Direct verification using \eqref{eq:KreinParam} confirms that for $k=1,2$, all pairs $(a,b)$ such that $q_{e_k,rs}^{ab}\neq 0$ satisfy $(a,b) \preceq_{(1,0)} e_k+(r,s)$ (see also Figure \ref{fig:Kreinparam}), 
    and if $e_k+(r,s)\in \cD^\star$ then $q_{e_k,rs}^{e_k+(r,s)}\neq 0$. By Proposition~\ref{prop:biQpoly2}, the result follows. 
\end{proof}
The bivariate polynomials $v^\star_{n}$ of Definition \ref{def:biQpoly} satisfy
\begin{equation}
U_{rs}(i,j)=v^\star_{rs}(U_{10}(i,j),U_{01}(i,j))\,, 
\end{equation} 
and hence 
\begin{align*}
    &xv^\star_{r,s}(x,y)=\sum_{\epsilon\in \{0,+1,-1\} }  \tcB^{\epsilon}(r,s)\ v^\star_{r+\epsilon,s}(x,y)\, , \\
    &xv^\star_{r,s}(x,y)=\sum_{\epsilon,\epsilon' \in \{0,+1,-1\}}  \tcC^{\epsilon,\epsilon'}(r,s)\ v^\star_{r+\epsilon,s+\epsilon'}(x,y)\,. 
\end{align*}
All the bivariate polynomials $v^\star_{r,s}(x,y)$ can be constructed from the initial conditions $v^\star_{00}=1$, $v^\star_{10}=x$ and $v^\star_{01}=y$ and the above recurrence relations. 

\section{Subconstituent algebra} \label{sec:subconstituent}

In this section, we explore the subconstituent algebra \cite{ter1,ter2,ter3} of the association schemes based on attenuated spaces following an approach similar to the one adopted for the non-binary Johnson schemes in \cite{CVZZ}. 

Recall that the vector space spanned by the adjacency matrices $A_{ij}$ of the scheme forms the commutative algebra $\cM$ under matrix product, with structure constants given by the intersection numbers $p_{ij,rs}^{ab}$. It is possible to construct an analogous algebra, called dual Bose--Mesner algebra, from a set of matrices derived from the idempotents $E_{rs}$,  with structure constants given by the Krein parameters $q_{ij,rs}^{ab}$; this goes as follows. 

Let $\mathbf{x}$ be a fixed element of the vertex set $X$ of the association scheme $\A$. For $(r,s)\in \cD^*$, the dual adjacency matrix $A^\star_{rs}$ (with respect to $\mathbf{x}$) is defined as the diagonal matrix with entries
\begin{align}\label{eq:Astar}
(A^\star_{rs})_{\mathbf{y}\mathbf{y}}&=|X| (E_{rs})_{\mathbf{x}\mathbf{y}}\,, \qquad \text{for all } \mathbf{y} \in X\,.
\end{align}
The vector space spanned by the dual adjacency matrices $A^\star_{rs}$ forms a commutative algebra with
\begin{equation}
 A^\star_{ij}  A^\star_{rs} = \sum_{(a,b)\in\cD^\star} q_{ij,rs}^{ab} A^\star_{ab} \,.
\end{equation}
This is the dual Bose--Mesner algebra (with respect to $\mathbf{x}$), that we denote $\cM^\star$.
For $(i,j)\in \cD$, the dual primitive idempotent $E_{ij}^\star$ (with respect to $\mathbf{x}$) is defined as the diagonal matrix with entries
\begin{align}
(E^\star_{ij})_{\mathbf{y}\mathbf{y}} &=(A_{ij})_{\mathbf{x}\mathbf{y}}\,, \qquad \text{for all } \mathbf{y} \in X\,.
\end{align}
These dual idempotents $E^\star_{ij}$ form another basis of $\cM^\star$ and satisfy
\begin{equation}
\sum_{(i,j)\in \cD} E_{ij}^\star = \II\,, \qquad E_{ij}^\star E_{mn}^\star = \delta_{im}\delta_{jn} E_{ij}^\star\,. 
\end{equation}
The dual adjacency matrices and dual idempotents are related as follows:
\begin{equation}
    A^\star_{rs} = \sum_{(i,j) \in \cD} U_{rs}(i,j) E_{ij}^\star\,, \qquad E^\star_{ij} =\frac{1}{|X|}\sum_{(r,s) \in \cD^\star} T_{ij}(r,s) A^\star_{rs}\,. 
\end{equation}

The subconstituent algebra (with respect to $\mathbf{x}$), also called the Terwilliger algebra, is defined as the algebra generated by $A_{ij}$ and $A^\star_{rs}$ for $(i,j)\in\cD$ and $(r,s)\in\cD^\star$. It is in general a non-commutative algebra, since the elements of the subalgebra $\cM$ do not necessarily commute with the elements of the subalgebra $\cM^\star$. The following lemma provides some useful relations which hold in the subconstituent algebra of an association scheme.
\begin{lem}\label{prop:EAE} \cite[Lemma 3.2]{ter1} For any fixed element $\mathbf{x} \in X$ we have
    \begin{align}
        & E_{ij}^\star A_{mn} E_{rs}^\star = 0 \quad \Leftrightarrow \quad p_{ij,mn}^{rs} =0\,, \quad (i,j),(m,n),(r,s) \in \cD \\
        & E_{ij} A_{mn}^\star E_{rs} = 0 \quad \Leftrightarrow \quad q_{ij,mn}^{rs} =0\,, \quad (i,j),(m,n),(r,s) \in \cD^\star\,. 
    \end{align}
\end{lem}

In the univariate case, the subconstituent algebra of a $P$- and $Q$-polynomial association scheme is generated by two elements $A,A^\star$ which satisfy the tridiagonal relations \cite{ter3}. For a bivariate $P$-polynomial association scheme, the Bose--Mesner algebra is generated by $A_{10}$ and $A_{01}$ because of \eqref{eq:polyA}. Similarly, for a bivariate $Q$-polynomial association scheme, one deduces from \eqref{eq:polyE} and \eqref{eq:Astar} that 
\begin{equation}
 A^\star_{mn}=v^\star_{mn}( A^\star_{10}, A^\star_{01}),
\end{equation}
hence the dual Bose--Mesner algebra is generated by $A^\star_{10}$ and $A^\star_{01}$. It follows that the subconstituent algebra of a bivariate $P$- and $Q$-polynomial association scheme such as $\A$ is generated by the four elements $A_{10}, A_{01}, A^\star_{10}, A^\star_{01}$, with $[A_{10},A_{01}]=[A^\star_{10},A^\star_{01}]=0$. In the next proposition, we obtain additional relations satisfied by the generators.
\begin{prop}\label{prop:subconstituent}
The following relations hold in the subconstituent algebra of the association schemes based on attenuated spaces $\A$.
The elements $A_{01}$ and $A^\star_{10}$ commute:
\begin{equation}\label{eq:relsub1}
[A_{01},A^\star_{10}]=0\,.
\end{equation}
The elements $A_{10}$ and $A^\star_{10}$ satisfy the tridiagonal relations
\begin{align}
&[A_{10},(q+q^{-1})A_{10}A^\star_{10}A_{10}-\{(A_{10})^2,A^\star_{10}\}+\gamma\{A_{10},A^\star_{10}\}+\rho A^\star_{10}]=0\,, \label{eq:relsub1}\\
&[A^\star_{10},(q+q^{-1})A^\star_{10}A_{10}A^\star_{10}-\{(A^\star_{10})^2,A_{10}\}+(1-q)\chi\{A^\star_{10},A_{10}\}+q\chi^2A_{10}]=0\,, \label{eq:relsub2}
\end{align}
where the elements $\gamma,\rho,\chi$ commute with $A_{10},A_{10}^\star$ and are given by
\begin{align}
\begin{split}
&\gamma=(1-q)\eta_1 +q^{\ell-1} \left(1+q^{n-m+1}\right)\,, \quad \rho=q\eta_1(\eta_1+\eta_0)+\frac{q^{2\ell-1} \left(1-q^{n-m+2}\right) \left(1-q^{n-m}\right)}{(1-q)^2}\,, \\
&\text{with} \quad \eta_0=q^{\ell-1}\left(\frac{2\left(1-q^{n-m-1}\right)}{1-q}+q^{n-m-1}(1-q)\right), \quad \eta_1= -\frac{1-q^m}{1-q}-A_{01}\,,
\end{split}\label{eq:gamrho}
\\
&\chi=q^{-m}\left(\frac{1-q^{n-1}}{1-q^{n-m}}\right)\left(1-q^m-\frac{1-q^n}{1-q^m}\right).
\end{align}
The elements $A_{01}$ and $A^\star_{01}$ satisfy the tridiagonal relations
\begin{align}
&[A_{01},(q+q^{-1})A_{01}A^\star_{01}A_{01}-\{(A_{01})^2,A^\star_{01}\}+(1-q)\zeta\{A_{01},A^\star_{01}\}+q\zeta^2A^\star_{01}]=0\,, \label{eq:relsub3}\\
&[A^\star_{01},(q+q^{-1})A^\star_{01}A_{01}A^\star_{01}-\{(A^\star_{01})^2,A_{01}\}+(1-q)\xi\{A^\star_{01},A_{01}\}+q\xi^2A_{01}]=0\,, \label{eq:relsub4}
\end{align}
where the elements $\zeta,\xi$ commute with $A_{01},A_{01}^\star$ and are given by
\begin{align}
&\zeta=q^{-1}(1-q^m-q^\ell)\,, \\
&\xi=q^{-1}\left(1-q^m-\frac{q^\ell \left(1-q^n\right)}{1-q^m}\right)
+q^{m-2}(1-q)\frac{ 1-q^{n-m}}{1-q^{n-1}}A^\star_{10}\,.
\end{align}
\end{prop}
\begin{proof}
Let us write the eigenvalues \eqref{eq:T10}--\eqref{eq:T01} and \eqref{eq:U10}--\eqref{eq:U01} as
\begin{equation}
\tau_{rs}=T_{10}(r,s)\,, \quad \mu_{s}=T_{01}(r,s)\,, \quad \tau^\star_{i}=U_{10}(i,j)\,, \quad \mu^\star_{ij}=U_{01}(i,j)\,.
\end{equation}
Note that $T_{01}(r,s)$ depends only on $s$ and $U_{10}(i,j)$ depends only on $i$, hence the notations $\mu_{s}$ and $\tau^\star_{i}$. By inserting the identity as a sum over the idempotents $E_{ij}$ on both sides and using the relation $A_{01}E_{ij}=\mu_{j}E_{ij}$, one finds
\begin{equation}
[A_{01},A^\star_{10}] = \sum_{(a,b),(c,d) \in \cD^\star}(\mu_b-\mu_d)E_{ab}A^\star_{10} E_{cd}\,.
\end{equation}
Using the values of the Krein parameters $q_{10,ab}^{cd}$, one deduces that the matrix $E_{ab}A^\star_{10} E_{cd}$ vanishes if $b\neq d$ (see Figure \ref{fig:recq10}). In the case where $b=d$, the factor $(\mu_b-\mu_d)$ vanishes. Therefore the previous sum is zero. 

Now we prove relation \eqref{eq:relsub1} by summing over the idempotents $E_{ij}$ again and using the relation
$A_{10}E_{ij}=\tau_{ij}E_{ij}$ (noting that the elements $\gamma,\rho$  commute with $E_{ij}$,  $A_{10}$ and $A^\star_{10}$):
\begin{align} \label{eq:A101}
&[A_{10},(q+q^{-1})A_{10}A^\star_{10}A_{10}-\{(A_{10})^2,A^\star_{10}\}+\gamma\{A_{10},A^\star_{10}\}+\rho A^\star_{10}] \nonumber \\
&\qquad =\sum_{(a,b),(c,d) \in \cD^\star}((q+q^{-1})\tau_{ab}\tau_{cd}-\tau_{ab}^2-\tau_{cd}^2+\gamma(\tau_{ab}+\tau_{cd})+\rho)(\tau_{ab}-\tau_{cd})E_{ab}A^\star_{10} E_{cd}\,,
\end{align} 
where in the second line, the occurrences of $A_{01}$ in $\gamma$ and $\rho$ are replaced by $\mu_b$. 
The only non-vanishing Krein parameters $q_{10,ab}^{cd}$ are those with $c=a-1,a,a+1$ and $d=b$ (see Figure \ref{fig:recq10}). The factor $(\tau_{ab}-\tau_{cd})$ in the previous sum obviously vanishes if $(a,b)=(c,d)$. 
Therefore we only need to consider the cases with $c=a\pm1$ and $d=b$. 
With the values of the elements $\gamma,\rho$ given as in \eqref{eq:gamrho}, the first factor in each of such summand indeed vanishes.  
Note that it is apparent that $\gamma$ should be of maximum degree 1 in $\mu_b$ and $\rho$ should be of maximum degree 2 in order for the first factor in the summand to possibly vanish. 

A similar approach can be used to prove relation \eqref{eq:relsub3}. In this case, we have an expression similar to \eqref{eq:A101}, with $\gamma$ and $\rho$ replaced by $(1-q)\zeta$ and $q\zeta^2$, respectively, $\tau_{ab}$ and $\tau_{cd}$ replaced by $\mu_b$ and $\mu_d$, and $E_{ab}A^\star_{10}E_{cd}$ replaced by $E_{ab}A^\star_{01}E_{cd}$. 
The other two relations are proved similarly, by multiplying $\mathbb{I}=\sum_{(r,s)\in \cD}E^\star_{rs} $ on both sides of the equations.
\end{proof}

\begin{rem}
Up to affine transformations, the operators $X,Y,X^\star,Y^\star$ acting on the vector space $W$ spanned by the polynomials $\{T_{ij}\}_{(i,j)\in\cD}$, as defined in Subsection \ref{ssec:bispalg}, form a representation of the matrices $A_{01},A_{10},A_{01}^\star,A_{10}^\star$. In other words, there is an algebra homomorphism from the subconstituent algebra to the bispectral algebra. The precise mappings of generators can be deduced from equation \eqref{eq:transformations} giving $X,Y,X^\star,Y^\star$ in terms of the bispectral operators $\cX,\cY,\cX^\star,\cY^\star$, which are defined by their actions \eqref{eq:XYs}--\eqref{eq:XY} on $T_{ij}$, and from equations \eqref{eq:T10}--\eqref{eq:T01} expressing the eigenvalues $T_{10}(r,s),T_{01}(r,s)$ in terms of $\lambda(s),q^{-s}\Lambda(r,s)$. Note that this representation corresponds to the one on the primary module \cite{ter1,TerCourse}. See also \cite{CVZZ} for a more detailed explanation in case of the non-binary Johnson schemes.  
\end{rem}

\begin{rem}
The notion of an $A_M$ multivariate $P$- and $Q$-polynomial association scheme on a domain $\cD=\{(n_1,n_2,\dots,n_M)\in\NN^M \ | \ n_1+n_2+\dots +n_M \leq N\}$ has recently been introduced in \cite{BKZZ2}. It is shown in \cite{BKZZ2} that the subconstituent algebra of such schemes, when acting on the primary module, has the structure of an $A_M$-Leonard pair \cite{IT}. In particular, the association schemes based on attenuated spaces are $A_2$ bivariate $P$- and $Q$-polynomial if their domain is triangular, \textit{i.e.}\ for $m \leq \min(n-m, \ell)$. Knowing that the subconstituent algebra on the primary module is equivalent to the bispectral algebra in Subsection \ref{ssec:bispalg}, we deduce that it has more precisely the structure of a factorized $A_2$-Leonard pair, see Definition 3.3 of \cite{CZ}. Indeed, this structure corresponds to one of the examples treated in \cite{CZ}. 
\end{rem}

\section{Relation to the non-binary Johnson scheme} \label{sec:limit}

The parameter $r$ for the non-binary Johnson scheme (see its definition below) is the size of the alphabet for the scheme. 
The parameter $q$ for association schemes based on attenuated spaces is the size of the underlying finite field, hence a prime power. Here we assume $q=p^h$ for some prime number $p$. 

It is asserted in the literature that association schemes based on attenuated spaces offer $q$-analogs of the non-binary Johnson schemes. Here we confirm this observation by first offering a combinatorial argument to embed certain non-binary Johnson schemes (when $r=q^\ell+1$ for some nonnegative integer $\ell$) into association schemes based on attenuated spaces.  
Second, we note that the parameters (such as the eigenvalues, dual eigenvalues, or cardinality of the vertex set) of association schemes based on attenuated spaces are well-defined functions for any $h\in \mathbb{R}$. Now with the parameter $\ell$ written in the form $\ell=\log_p(r-1)/h$, connecting the parameters $r$ and $h$ of the two schemes, we see that upon taking the limit $h\rightarrow0$ ($q\rightarrow 1$), all the relevant parameters of association schemes based on attenuated spaces approach the corresponding parameters of the non-binary Johnson schemes. 

\subsection{Embedding of non-binary Johnson schemes}

Here we provide a combinatorial argument in support of the assertion that the scheme $\A$ is a $q$-analog of the non-binary Johnson scheme by showing that there exists an embedding of the latter in $\A$ if $r-1=q^\ell$. First we review the definition of non-binary Johnson scheme. 

Let $ K = \{0,1, \dots, r-1\}$ be an alphabet of cardinality $r>2$ and consider its $n$-th Cartesian power $K^n$. For $\vx\in K^n$, the weight $w(\vx)$ is defined to be the number of its nonzero components. Fix a positive integer $m$ and let $\mathfrak{X}$ be the subset of $K^n$ consisting of elements of weight $m$:
\begin{equation}
	\mathfrak{X} = \{ \vx \in K^n \ | \ w(\vx)=m \}\,. 
\end{equation} 
Then $\mathfrak{X}$ has cardinality
\begin{equation}
   |\mathfrak{X}| = \binom{n}{m}(r-1)^m \, . 
\end{equation}
The set of non-empty binary relations on $\mathfrak{X}$ 
\begin{equation}
    \mathfrak{R}_{ij} = \{(\vx ,\vy) \in  \mathfrak{X}\times  \mathfrak{X} \ |\ e(\vx ,\vy)  = m-i-j, \ c(\vx ,\vy)  = m-i\}
\end{equation}
with 
\begin{equation}
    e(\vx ,\vy) = |\{ i\  | \ \vx_i = \vy_i \neq 0\}|\,, \quad  c(\vx ,\vy) = |\{ i \ |\  \vx_i \neq 0, \  \vy_i \neq 0\}|
\end{equation}
forms a symmetric association scheme on $\mathfrak{X}$, called the non-binary Johnson scheme \cite{Dun,TAG}, denoted here by $J_{r}(n,m)$.  

Now, given an association scheme based on attenuated spaces $\A$, consider the non-binary Johnson scheme with  $r = q^\ell + 1$ so that $K\backslash \{0\}$ and the subspace $\mathbf{w} \subset \mathbb{F}_q^{n+\ell}$ of dimension $\ell$ share the same cardinality.  In particular, one can define a bijection $\phi: K\backslash \{0\} \rightarrow \mathbf{w}$ between these two sets. Let us further fix a set of $n$ linearly independent vectors  $\{f_i\}_{1\leq i \leq n}$ in $\mathbb{F}_q^{n+\ell}$ whose span  intersects $\mathbf{w}$ trivially. We can then define the following map $\varphi: \mathfrak{X} \rightarrow X$, from the vertex set $\mathfrak{X}$ of $J_{r}(n,m)$ to the vertex set $X$ of $\A$,
\begin{equation}
    \varphi(\vx) = \text{span}\{ f_i + \phi(\vx_i)  \ | \ \forall \ i  \text{ such that}\ \vx_i \neq 0\}\,.
\end{equation}
Two notable observations can be made on $\varphi$. First, stemming from the linear independence of the vectors $f_i$ and the injective nature of $\phi$, it follows that $\varphi$ is likewise injective. Second, we note that
\begin{equation}
  \dim(( \varphi(\vx)+\mathbf{w})/\mathbf{w} \cap ( \varphi(\vy)+\mathbf{w})/\mathbf{w} ) = c( \vx, \vy)\,,
\end{equation}
\begin{equation}
 \dim( \varphi(\vx) \cap  \varphi(\vy)) = e(\vx, \vy)\,,
\end{equation}
so that $\varphi$ maps pairs of elements in $\mathfrak{X}$ in relation $\mathfrak{R}_{ij}$  to pairs of elements in $X$ in relation $R_{ij}$, \textit{i.e.}
\begin{equation}
    (\vx, \vy) \in \mathfrak{R}_{ij} \Rightarrow (\varphi(\vx), \varphi(\vy)) \in R_{ij}\,.
\end{equation}
Hence, one concludes that $\varphi$ indeed provides an embedding of $J_{r}(n,m)$ in $\A$. 

\subsection{Limit for the polynomial structure}

An alternative way to talk of $q$-analogs for association schemes is by considering the $q \rightarrow 1$ limit of its defining quantities. 
Putting aside the combinatorial meaning of the parameters such as $q=p^h$ or $\ell$ of association schemes based on attenuated spaces, we see 
from equations \eqref{eq:cardinality} and \eqref{eq:affKrawscheme}--\eqref{eq:Urs} that the cardinality $|X|$, the eigenvalues $T_{ij}$,  and the dual eigenvalues $U_{ij}$ are well defined functions for any $h,\ell\in \mathbb{R}$. 
Here we show that for special parametrization of $\ell$, connecting the parameters $r$ and $h$ of the two schemes, all the functions such as $T_{ij}$ of association schemes based on attenuated spaces approach the ones of the non-binary Johnson schemes.

First recall that the eigenvalues and dual eigenvalues of $J_r(n,m)$ are respectively given by (see \cite{TAG,CVZZ})
\begin{align}
&\tT_{ij}(x,y) = (r-1)^j\tK_{i}(m-j,r-1,x)\tE_{j}(n-x,m-x,y)\,, \label{eq:tTij} \\
&\tU_{ij}(x,y) = \frac{\binom{n}{i}}{\binom{m}{i}}\tK_{i}(m-y,r-1,x)\tQ_{j}(n-i,m-i,y)\,, \label{eq:tUij}
\end{align}
where 
\begin{align}
&\tK_{i}(m-j,r-1,x) = \binom{m-j}{i}(r-2)^i \ \hat{K}_i(x;(r-2)/(r-1),m-j)\,,\\
&\tE_{j}(n-x,m-x,y) = \binom{m-x}{j} \binom{n-m}{j} R_j(\lambda(y);-m+x-1,-n+m-1,n-m)\,,\\
&\tQ_j(n-i,m-i,y) =  \left(\binom{n-i}{j}-\binom{n-i}{j-1}\right)R_y(\lambda(j);-m+i-1,-n+m-1,n-m)\,,
\end{align}
with $\hat{K}_i(x)$ the Krawtchouk polynomials and $R_j(\lambda(y))$ the dual-Hahn polynomials (see Appendix \ref{app:notations}). 

Let $\ell=\log_p(r-1)/h$, where $q=p^h$, then 
\begin{align}
    &\lim_{h\to 0} q =\lim_{h\to 0} p^h =1\, , \\
    &\lim_{h\to 0} \qbinom{n}{m}=\binom{n}{m}\, ,\\
    &\lim_{h\to 0} q^\ell=\lim_{h\to 0} p^{h \log_p(r-1)/h}=r-1\,.
\end{align}
With these limits, it can be shown straightforwardly from the expressions \eqref{eq:cardinality}, and 
\eqref{eq:affKrawscheme}--\eqref{eq:Hahnscheme} that:
\begin{align}
& \lim_{ h \to 0} |X| =  \lim_{h \to 0} \qbinom{n}{m}q^{\ell m} =  \binom{n}{m}(r-1)^m  = |\mathfrak{X}|\, , \\
&\lim_{h\to 0} K_i(m-j,\ell;q;x)= \tK_{i}(m-j,r-1,x)\,,\\
&\lim_{h\to0} E_j(n-x,m-x;q;y) = \tE_{j}(n-x,m-x,y)\,,\\
&\lim_{h\to0} Q_j(n-i,m-i;q;y) = \tQ_j(n-i,m-i,y)\,.
\end{align}
Therefore one deduces from \eqref{eq:Tij}, \eqref{eq:Urs}, \eqref{eq:tTij} and \eqref{eq:tUij} that
\begin{align}
\lim_{h\to 0}T_{ji}(y,x)= \tT_{ij}(x,y)\,, \qquad \lim_{h\to 0}U_{ji}(y,x)= \tU_{ij}(x,y)\,,
\end{align}
that is, that the cardinality, the eigenvalues and the dual eigenvalues of the association schemes based on attenuated spaces $\A$ correspond respectively in the limit $q^\ell \to r-1$ and $q\to1$ to the cardinality, the eigenvalues and the dual eigenvalues of the non-binary Johnson schemes $J_r(n,m)$. Other parameters of $J_r(n,m)$ can be similarly recovered from those of $\A$, such as the intersection numbers, the Krein parameters and the relations of the subconstituent algebra (for the latter, compare Proposition \ref{prop:subconstituent} in the present paper to Proposition 5.4 in \cite{CVZZ}, with the appropriate change of notations), or more directly, all these parameters or relations are determined by the eigenvalues of the scheme.  

Let us finally highlight that the limit $q\to1$ of $\A$ with all parameters $n,\ell,m$ fixed ($\ell$ is independent of the parameters $r$ and $h$) 
corresponds to the limit $r\to2$ of $J_r(n,m)$, which is the binary Johnson scheme.

\section{Outlook} \label{sec:outlook}

This paper examined the bivariate $P$- and $Q$-polynomial structures of association schemes based on attenuated spaces. Although these structures have been characterized in recent works \cite{BCPVZ,BKZZ,BKZZ2}, the present paper features an original approach which is based on the determination of the recurrence relations and difference equations of bivariate polynomials from the properties of the constituant univariate polynomials. The intersection numbers and Krein parameters of the schemes were then identified 
from the coefficients of the bispectral equations of the eigenvalues, without need of combinatorial arguments. 
The bivariate $P$- and $Q$-polynomial structure of the scheme was further shown to satisfy a more refined structure, $\poo$-compatible for certain $\alpha, \beta$. As shown in Section \ref{sec:bivariatePQpoly}, a lot of information is contained in this compatibility, in particular the degree of recurrence or difference relations satisfied by the scheme, or by the associated bivariate polynomials. 
The approach put forward in this paper naturally led to the exploration of the bispectral algebra associated to the bivariate polynomials which form the eigenvalues. Algebraic relations were obtained as well for the subconstituent (or Terwilliger) algebra of the association scheme. Finally, the connection between the Johnson schemes and the schemes based on attenuated spaces was made explicit, and the properties of both were compared.     

Other higher rank association schemes whose eigenvalues can be expressed in terms of univariate polynomials could benefit from the methods presented in this paper, which focus on the properties of polynomials rather than specific combinatorial aspects. Among these are those based on isotropic subspaces, which are related to the univariate $q$-dual Hahn polynomial.

Concerning the algebraic aspects, the example of association schemes based on attenuated spaces studied in the present paper (together with its limit, the non-binary Johnson schemes studied in \cite{CVZZ}) provides grounds for further exploration of the higher rank algebras associated to multivariate orthogonal polynomials. It would thus be valuable to get a complete description of the bispectral algebra investigated in Subsection \ref{ssec:bispalg}, or of the subconstituent algebra in Section \ref{sec:subconstituent}. 
For the latter, an additional relation between $A_{10}$ and $A_{01}^*$ could exist, generalizing the tridiagonal relations derived in Proposition \ref{prop:subconstituent}. These tridiagonal relations are central in univariate $P$- and $Q$-polynomial schemes, and identifying a suitable generalization is crucial for advancing our understanding of the subconstituent algebra of multivariate $P$- and $Q$-polynomial association schemes. 

An intermediate step towards these goals would be to focus on a subclass of association schemes with properties similar to those based on attenuated spaces. In terms of $\poo$-compatible bivariate $P$- and $Q$-polynomial association schemes, this corresponds to the subclass with either $\alpha=0$ and $\beta=1$ or $\alpha=1$ and $\beta=0$, so that the pair of recurrence relations satisfied by the eigenvalues (or dual eigenvalues) have respectively three and nine terms. In fact, this subclass of schemes could be restrained further so that the latter recurrence relation contains at most seven terms, as is actually the case for the schemes based on attenuated spaces. On triangular domains $\cD$, such a subclass would correspond to a refinement of the concept of $A_2$ bivariate $P$- and $Q$-polynomial association scheme \cite{BKZZ2}. It would be interesting then to determine if such schemes lead to factorized $A_2$-Leonard pairs \cite{CZ}, in the same way that $A_M$ multivariate $P$- and $Q$-polynomial association schemes lead more generally to $A_M$-Leonard pairs \cite{IT,BKZZ2}.

\vspace{1cm}
\noindent \textbf{Acknowledgements.}
PAB and MZ hold an Alexander-Graham-Bell scholarship from the Natural Sciences and Engineering Research Council (NSERC) of Canada.
NC thanks the CRM for its hospitality and is supported by the international research project AAPT of the CNRS. 
The research of LV is supported in part by a Discovery Grant from the NSERC. 

\appendix

\section{Standard notations}\label{app:notations}

We record here for reference the following definitions.

\subsection{Basic hypergeometric functions}
The $q$-Pochhammer symbol:
\begin{equation}
(a;q)_0:=1\,, \qquad (a;q)_n=\prod_{k=0}^{n-1}(1-aq^k)\,, \quad n = 1,2,\dots 
\end{equation}
The $q$-binomial coefficient:
\begin{equation}
\qbinom{n}{k} = \frac{(q;q)_n}{(q;q)_k(q;q)_{n-k}} \,.
\end{equation}
The basic hypergeometric function:
\begin{equation}
{}_r\phi_s \Biggl({{a_1,\; \dots, \;a_r}\atop
{b_1,\; \dots, \; b_s}}\;\Bigg\vert \; q,z\Biggr) = \sum_{k=0}^\infty \frac{(a_1;q)_k \cdots (a_r;q)_k}{(b_1;q)_k \cdots (b_s;q)_k}(-1)^{(1+s-r)k}q^{(1+s-r)\binom{k}{2}}\frac{z^k}{(q;q)_k}\,.
\end{equation}
The affine $q$-Krawtchouk polynomials:
\begin{equation}\label{eq:affqK}
K_n^{aff}(q^{-x};p,N;q)=    {}_3\phi_2 \Biggl({{q^{-n},\; 0,\;q^{-x}}\atop
{pq,\; q^{-N}}}\;\Bigg\vert \; q,q\Biggr) \,, \quad n=0,1,\dots,N \,.
\end{equation}
The dual $q$-Hahn polynomials:
\begin{equation}\label{eq:dH}
R_n(\mu(x);\gamma,\delta,N|q)=    {}_3\phi_2 \Biggl({{q^{-n},\;\gamma\delta q^{x+1}}, \; q^{-x}\atop
{\gamma q,\; q^{-N}}}\;\Bigg\vert \; q,q\Biggr) \,, \quad n=0,1,\dots,N \, ,
\end{equation}
where $\mu(x)=q^{-x}+\gamma\delta q^{x+1}$.

\subsection{Hypergeometric functions}
The Pochhammer symbol:
\begin{equation}
(a)_0:=1\,, \qquad (a)_n=\prod_{k=0}^{n-1}(a+k)\,, \quad n = 1,2,\dots 
\end{equation}
The hypergeometric function:
\begin{equation}
{}_rF_s \Biggl({{a_1,\; \dots, \;a_r}\atop
{b_1,\; \dots, \; b_s}}\;\Bigg\vert \; z\Biggr) = \sum_{k=0}^\infty \frac{(a_1)_k \cdots (a_r)_k}{(b_1)_k \cdots (b_s)_k}\frac{z^k}{k!}\,.
\end{equation}
The Krawtchouk polynomials:
\begin{align}
& \hat{K}_n(x;p,N)= {_2}F_1\Biggl({{-n,-x} \atop {-N}}
\;\Bigg|\; \frac{1}{p} \Biggr)\,, \qquad n=0,1,\dots,N\,.
\end{align} 
The dual-Hahn polynomials:
\begin{align}
& R_n(\lambda(x);\gamma,\delta,N)={_3}F_2\Biggl({{-n,-x,x+\gamma+\delta+1}\atop {\gamma+1,-N}} \; \Bigg| \; 1 \Biggl)\,, \qquad n=0,1,\dots,N\,,
\end{align} 
where $\lambda(x)=x(x+\gamma+\delta+1)$.

\section{Relations for the univariate polynomials} \label{app:relations} 

The recurrence relations and difference equations of the univariate polynomials $K_j(N,\ell;q;s)$ and $E_{i}(N,m;q;r)$ (see \textit{e.g.}\ \cite{Koek}) are recorded below.
We provide also the contiguity recurrence relations for $K_j(N,\ell;q;s)$ and the contiguity difference equations for $E_{i}(N,m;q;r)$. These relations can be proven by induction on the degree or the variable of the polynomials (see \cite{CFR,CZ}). 

\paragraph{(Contiguity) recurrence relation for $K_j(N,\ell;q;s)$.}

For $\epsilon\in\{0,+1,-1\}$, $0 \leq j \leq \min(N,\ell)$ and $0 \leq s \leq \min(N,N+\epsilon,\ell)$
one gets
\begin{align} \label{eq:crecK}
 \lambda^\epsilon(s) K_j(N,\ell;q;s)=& \sum_{\epsilon'\in \{0,+1,-1\} }  a_j^{\epsilon,\epsilon'} K_{j+\epsilon'}(N+\epsilon,\ell;q;s) \,,
\end{align}
where 
\begin{align}
& \lambda^+(s)=(q^{-s}-q^{-N-1})q^{\ell+1}\,,&&  a_j^{++}=(q^{j+1}-1)q^{-N}   \,, \quad a_j^{+-}=0\,,\nonumber \\
&&& a_j^{+0}=q-q^{j-N}\,, \\
 &\lambda^0(s)=q^{-s}-1\,, &&a_j^{0+}=(q^{j+1}-1)q^{j-N-\ell}\,,\quad  a_j^{0-}=(1-q^{j-N-1})(1-q^{j-\ell-1})\,,\nonumber \\
 &&& a_j^{00}=-a_{j-1}^{0+}-a_{j+1}^{0-} \,, \\
 & \lambda^-(s)=1\,,&& a_j^{-+}=0\,,\quad  a_j^{--}= q^\ell-q^{j-1}\,, \nonumber\\
&&& a_j^{-0}=q^j \,.
\end{align}

\paragraph{Difference equation for $K_j(N,\ell;q;s)$.}
For $ 0 \leq j,s\leq \min(N,\ell)$, 
the following difference relation holds:
\begin{align} \label{eq:diffK}
 \theta_j K_j(N,\ell;q;s)=& \sum_{\epsilon \in \{0,+1,-1\} }  b^{\epsilon}(s) K_{j}(N,\ell;q;s+\epsilon) \,,
\end{align}
where 
\begin{align}
& \theta_j=q^{-j}-1 \,,&& b^{+}(s)=(1-q^{s-N})(1-q^{s-\ell})   \,, \quad b^{-}(s)= -q^{s-\ell-N-1}(1-q^s)  \,,\nonumber\\
&&&b^{0}(s)=- b^{+}(s)- b^{-}(s) \,.
\end{align}

\paragraph{Recurrence relation for $E_i(N,m;q;r)$.}

For $ 0 \leq i,r \leq \min(N-m,m)$, 
the polynomial $E_i(N,m;q;r)$ satisfies the following recurrence relation:
\begin{align}\label{eq:recE}
 \Lambda(r) E_i(N,m;q;r)=& \sum_{\epsilon\in \{0,+1,-1\} }  A_i^{\epsilon}\ E_{i+\epsilon}(N,m;q;r) \, ,
\end{align}
where 
\begin{align}
& \Lambda(r)=(q^{-r}-1)(1-q^{r-N-1})\,,&&  A_i^{+}=(1-q^{i+1})^2q^{-N-1}   \,, \quad A_i^{-}=(1-q^{i-1-N+m})(1-q^{i-m-1})\,,
 \nonumber\\
&&& A_i^{0}=-A_{i-1}^{+}-A_{i+1}^{-}\,.
\end{align}

\paragraph{(Contiguity) difference equations for $E_i(N,m;q;r)$.}

For $\epsilon\in\{0,+1,-1\}$ and $ 0 \leq i,r \leq \min(N-m,m)$, 
one gets
\begin{align}\label{eq:cdiffE}
 \Theta^\epsilon_i E_i(N,m;q;r)=& \sum_{\epsilon' \in \{0,+1,-1\} }  B^{\epsilon,\epsilon'}(r) E_i(N+\epsilon,m+\epsilon;q;r+\epsilon')\,,
\end{align}
where 
\allowdisplaybreaks
\begin{align}
&  \Theta^+_i=1  \,,&& B^{++}(r)=\frac{1-q^{r+m-N}}{1-q^{2r-N-1}}
   \,, \quad B^{+-}(r)= 0 \,,\nonumber \\
&&& B^{+0}(r)=-\frac{q^{2r-N-1}(1-q^{m+1-r})
}{1-q^{2r-N-1}} \,,\\
&  \Theta^0_i=q^{-i}-1  \,,&& B^{0+}(r)=\frac{(1-q^{r-N+m})(1-q^{r-N-1})(1-q^{r-m})}{(1-q^{2r-N-1})(1-q^{2r-N})}\,,\nonumber\\
 &&&  B^{0-}(r)= - q^{r-N-1}\, \frac{(1-q^r)(1-q^{r-m-1})(1-q^{r-N+m-1})}{(1-q^{2r-N-1})(1-q^{2r-N-2})}  \,,\nonumber \\
&&& B^{00}(r)=-B^{0+}(r)-B^{0-}(r) \,, \label{eq:diffE}\\
&  \Theta^-_i=q^{m-i} -1 \,,&& B^{-+}(r)=0   \,, \quad 
B^{--}(r)= -\frac{(1-q^{r-N+m-1})(1-q^r)}{1-q^{2r-N-1}}\,,\nonumber \\
&&& B^{-0}(r)=\frac{(1-q^{r-N-1})(q^m-q^r)}{1-q^{2r-N-1}} \,.
\end{align}

\end{document}